\documentclass[10 pt, a4paper, american, preprint]{article}
\usepackage[english]{babel}
\usepackage{amsmath}             
\usepackage{amssymb}             
\usepackage{mathtools}
\usepackage{amsthm}
\usepackage{latexsym}            
\usepackage[active]{srcltx}
\usepackage{ae}
\usepackage{color}
\usepackage{geometry}
\usepackage{hyperref}
\hypersetup{colorlinks=true, linkcolor=blue, citecolor=blue}
\numberwithin{figure}{section}
\usepackage[utf8]{inputenc}
\usepackage{dsfont}
\usepackage{xcolor}
\usepackage{nicefrac}
\allowdisplaybreaks

\newtheorem{theorem}{Theorem}[section]
\newtheorem{lemma}[theorem]{Lemma}
\newtheorem{proposition}[theorem]{Proposition}

\theoremstyle{definition}
\newtheorem{definition}[theorem]{Definition}

\newtheorem{remark}[theorem]{Remark}
\numberwithin{equation}{section}

\def\vint_#1{\mathchoice%
          {\mathop{\kern 0.2em\vrule width 0.6em height 0.69678ex depth -0.58065ex
                  \kern -0.8em \intop}\nolimits_{\kern -0.4em#1}}%
          {\mathop{\kern 0.1em\vrule width 0.5em height 0.69678ex depth -0.60387ex
                  \kern -0.6em \intop}\nolimits_{#1}}%
          {\mathop{\kern 0.1em\vrule width 0.5em height 0.69678ex depth -0.60387ex
                  \kern -0.6em \intop}\nolimits_{#1}}%
          {\mathop{\kern 0.1em\vrule width 0.5em height 0.69678ex depth -0.60387ex
                  \kern -0.6em \intop}\nolimits_{#1}}}
\def\vintslides_#1{\mathchoice%
          {\mathop{\kern 0.1em\vrule width 0.5em height 0.697ex depth -0.581ex
                  \kern -0.6em \intop}\nolimits_{\kern -0.4em#1}}%
          {\mathop{\kern 0.1em\vrule width 0.3em height 0.697ex depth -0.604ex
                  \kern -0.4em \intop}\nolimits_{#1}}%
          {\mathop{\kern 0.1em\vrule width 0.3em height 0.697ex depth -0.604ex
                  \kern -0.4em \intop}\nolimits_{#1}}%
          {\mathop{\kern 0.1em\vrule width 0.3em height 0.697ex depth -0.604ex
                  \kern -0.4em \intop}\nolimits_{#1}}}

\newcommand{\aveint}[2]{\mathchoice%
          {\mathop{\kern 0.2em\vrule width 0.6em height 0.69678ex depth -0.58065ex
                  \kern -0.8em \intop}\nolimits_{\kern -0.45em#1}^{#2}}%
          {\mathop{\kern 0.1em\vrule width 0.5em height 0.69678ex depth -0.60387ex
                  \kern -0.6em \intop}\nolimits_{#1}^{#2}}%
          {\mathop{\kern 0.1em\vrule width 0.5em height 0.69678ex depth -0.60387ex
                  \kern -0.6em \intop}\nolimits_{#1}^{#2}}%
          {\mathop{\kern 0.1em\vrule width 0.5em height 0.69678ex depth -0.60387ex
                  \kern -0.6em \intop}\nolimits_{#1}^{#2}}}

\newcommand{\eps}{\varepsilon}
\newcommand{\R}{\mathds{R}}

\newcommand{\N}{\mathbb{N}}

\renewcommand{\liminf}{\operatornamewithlimits{lim \, inf}}

\newcommand{\supp}{\operatorname{supp}}

\renewcommand{\div}{\nabla \cdot}

\numberwithin{equation}{section}

\def\dx{\, {\mathrm d}x}

\def\dt{\, {\mathrm d}t}

\def\ds{\, {\mathrm d}s}
\def\d{\, {\mathrm d}}

\def\loc{\mathrm{loc}}

\def\Leb{\mathcal{L}}

\def\d{\, {\mathrm d}}

\def\BV{BV}

\newcommand{\Om}{\Omega}

\newcommand{\limminus}{{\mathchoice{\raise.17ex\hbox{$\scriptstyle -$}}
                {\raise.17ex\hbox{$\scriptstyle -$}}
                {\raise.1ex\hbox{$\scriptscriptstyle -$}}
                {\scriptscriptstyle -}}}
\newcommand{\limplus}{{\mathchoice{\raise.17ex\hbox{$\scriptstyle +$}}
                {\raise.17ex\hbox{$\scriptstyle +$}}
                {\raise.1ex\hbox{$\scriptscriptstyle +$}}
                {\scriptscriptstyle +}}}

\newcommand{\xint}[3]{\,{\setbox0=\hbox{$#1{#2#3}{\int}$}
\vcenter{\hbox{$#2#3$}}\kern-.5\wd0}}
\newcommand{\mint}{\mathchoice
{\xint\displaystyle\textstyle-}
{\xint\textstyle\scriptstyle-}
{\xint\scriptstyle\scriptscriptstyle-}
{\xint\scriptscriptstyle\scriptscriptstyle-}
\!\int}

\newcommand{\Lip}{\operatorname{Lip}}
\newcommand{\beq}{\begin{equation}}
\newcommand{\eeq}{\end{equation}}

\newcommand{\arsinh}[1]{\text{arsinh(}}

\renewcommand{\div}{\text{div}}
\newcommand{\I}{\textbf{I}}
\newcommand{\II}{\textbf{II}}
\newcommand{\III}{\textbf{III}}
\newcommand{\IV}{\textbf{IV}}

\numberwithin{equation}{section}

\newcommand{\X}{\mathcal{X}}

\renewcommand{\N}{\mathcal{N}}
\newcommand{\NN}{\mathds{N}}
\newcommand{\K}{\mathcal{K}}

\renewcommand{\L}{\mathcal{L}}

\newcommand{\weak}{\rightharpoondown}

\newcommand{\xweak}{\xrightharpoondown}

\newcommand{\longweak}{\xweak{~~~}}

\renewcommand{\BV}{\mathrm{BV}}
\newcommand{\Der}{\mathrm{Der}}
\newcommand{\Omstar}{\Om^*}
\newcommand{\Feps}{\mathcal{F}_{\varepsilon}}
\newcommand{\frd}{\mathfrak{d}}

\makeatletter
\newcommand{\subjclass}[2][1991]{
  \let\@oldtitle\@title
  \gdef\@title{\@oldtitle\footnotetext{#1 \emph{Mathematics subject classification.} #2}}
}
\newcommand{\keywords}[1]{
  \let\@@oldtitle\@title
  \gdef\@title{\@@oldtitle\footnotetext{\emph{Key words and phrases.} #1.}}
}
\makeatother

\title{Existence of parabolic minimizers to the total variation flow on metric measure spaces}

\author{Vito Buffa\footnote{Bologna, Italy; e-mail: \texttt{bff.vti@gmail.com}; \textsf{ORCID iD}: 0000-0003-4175-4848.}, Michael Collins\footnote{Department of Data Science, Friedrich-Alexander-Universit\"at  Er\-langen-N\"urnberg, Cauerstrasse 11, 91058 Erlangen, Germany; e-mail: \texttt{collins@math.fau.de}.}, Cintia Pacchiano Camacho\footnote{Department of Mathematics and Systems Analysis, Aalto University, Espoo, Finland; e-mail: \texttt{cintia.pacchiano@aalto.fi}.}}

\date{}


\begin{document}

\maketitle

\begin{abstract}\noindent
We give an existence proof for variational solutions $u$ associated to the total variation flow. Here, the functions being considered are defined on a metric measure space $(\X, d, \mu)$ satisfying a doubling condition and supporting a Poincar\'e inequality. For such parabolic minimizers that coincide with a time-independent Cauchy-Dirichlet datum $u_0$ on the parabolic boun\-dary of a space-time-cylinder $\Om \times (0, T)$ with $\Om \subset \X$ an open set and $T > 0$, we prove existence in the weak parabolic function space $L^1_w(0, T; \BV(\Om))$. In this paper, we generalize results from a previous work by B\"ogelein, Duzaar and Marcellini by introducing a more abstract notion for $\BV$-valued parabolic function spaces. We argue completely on a variational level.
\end{abstract}

\bigskip

\tableofcontents

\vspace{1cm}

\section{Introduction}

The aim of this paper is to show existence for parabolic minimizers to the total variation flow on metric measure spaces. More precisely, we consider minimizers of integral functionals that are related to scalar functions $u: \Omega \times (0, T) \to \R$ which satisfy the inequality
\begin{align}\label{introquasi}
\iint_{\Omega_T} u\partial_t\varphi\d\mu\dt + \int_0^T \|Du(t)\|(\Omega)\dt \leq \int_0^T \|D(u + \varphi)(t)\|(\Omega)\dt,
\end{align}
for all test functions $\varphi \in \Lip_c(\Om_T)$ where $\|Du(t)\|(\Omega)$ denotes the total variation of $u(\cdot, t)$ on $\Omega$. Here, $\Omega \subset \X$ is a bounded domain, where $(\X, d, \mu)$ is a metric measure space with a metric $d$ and a measure $\mu$.

In the setting of a metric measure space, the classical calculus known from the Euclidean space $\R^n$ is no longer available and instead of distributional derivatives, the space $\BV$ of functions with bounded variation has to be introduced by a relaxation approach \cite{MirandaGoodSpaces} that makes use of the notion of upper gradients. An alternative approach to $\BV$ via derivations \cite{DiMarino} is also presented.

This paper deals with parabolic minimizers on parabolic cylinders $\Om_T := \Om \times (0, T)$ with $\Om \subset \X$ bounded and open and $T > 0$. $\X$ denotes a metric measure space that fulfills a doubling property with respect to the metric $d$ and the measure $\mu$ and supports a suitable Poincar\'e inequality. We refer to Section \ref{sec-setting} for exact definitions and the setting of the relevant spaces. In this paper, we are going to generalize results which have recently been proven in \cite{boegeleintv}, while we restrict ourselves to the simplest case where the functional in question depends only on the total variation itself.

Since the beginning of the 21$^\text{st}$ century, doubling metric measure spaces have been studied quite extensively, see for example \cite{Cheeger,Franchi,Hajlasz1,Hajlasz2,HajlaszKoskela,heinonenkoskela,Kronz,Shanmugalingam,ShanmugalingamHarmonic} and especially \cite{Bjoern} for an overview and further references. The idea of considering variational problems on metric measure spaces is based on independent proofs by Grigor'yan \cite{grig} and Saloff-Coste \cite{saloffcoste} of the fact that on Riemannian manifolds the doubling property and Poincar\'e inequality are equivalent to a certain Harnack-type inequality for solutions of the heat equation. Instead of Riemannian manifolds though, we are interested in more general spaces. 

Existence for parabolic problems on metric measure spaces has already been dealt with in \cite{CollinsHeran} and \cite{Collins}, where the former paper treats boundary data independent from time and the latter treats time-dependent boundary data. In both cases, the authors considered integral functionals with $p$-growth for $p > 1$. This paper deals with the total variation flow which corresponds to the case $p = 1$. In the elliptic case, existence for functions of least gradient has been considered by Korte, Lahti, Li and Shanmugalingam in \cite{LeastGradients}. The investigation of parabolic problems on metric measure spaces started not long ago with the work of Kinnunen, Marola, Miranda and Paronetto, \cite{KinnunenMarola}, concerning regularity problems. Since then, the most contributions in this field of research have been made to stability theory,   \cite{habnew2,hab1,MaasaloZatorska}, and regularity problems, \cite{habnew3,HabermannIntegrability,MarolaMasson,masson2,masson}. When concerning issues of regularity, one tries to establish for instance H\"older continuity of a solution that is assumed to be an element of the parabolic Newtonian space $L^p(0, T; \N^{1,p}(\Om))$. By that function space, we denote those $u: (0, T) \to \N^{1,p}(\Om)$, such that the mapping $t \mapsto \|u(t)\|_{\N^{1,p}(\Om)}^p$ is integrable over the interval $(0, T)$.

When considering the total variation flow, one cannot simply set $p = 1$ in the parabolic function space since the Newtonian space $\N^{1,1}$ lacks important properties such as reflexivity. Therefore, it is replaced by $\BV$, the space of functions with bounded variation. Choosing the space $L^1(0, T; \BV(\Om))$ still would not be appropriate for our tasks though, since Bochner measurability is too restrictive to ask. To this end, we are going to consider a weak version of this parabolic space, denoted by $L^1_w(0, T; \BV(\Om))$. The Bochner measurability for Banach space valued functions is going to be replaced by a weaker measurability condition that makes use of the derivation approach for $\BV$. See Section \ref{sec-setting} for details.

Our aim is to show existence for a parabolic minimizer of the functional
\begin{align*}
\mathcal{F}(v, \Omstar_T) := \int_0^T\|Dv(t)\|(\Omstar)\dt
\end{align*}
in such a parabolic Newtonian space, which will be done via the concept of global variational solutions. The integrand denotes the total variation of $v$ on the time-slice $t \in (0, T)$ and $\Omstar$ is a bounded and open domain in $\X$, slightly larger than $\Om$. Again, we refer the reader to Section \ref{sec-setting} for the exact definitions and in particular to Theorem \ref{maintheo1} and Theorem \ref{maintheo2}, respectively, to find the exact statements.

In the paper at hand, we are going to put the focus on how to overcome certain difficulties given by the setting of metric measure spaces. The main difficulty that leads to such obstacles is given by the fact that the standard definition of the space $\BV$ on a metric measure space does not rely on an integration by parts formula and therefore $\BV$ cannot be characterized as the dual space of a separable Banach space as suggested by \cite[Remark 3.12]{AmbFusPal}. This makes it difficult to find a weak measurability condition similar to the one posed in \cite{boegeleintv} and other works concerning the total variation flow and functionals with linear growth. To overcome this obstacle and to be able to give a suitable definition of a parabolic function space, we make use of an alternative approach: in the Euclidean case, a function $u: (0, T) \to \BV(\Om)$ is said to be weakly measurable if the mapping
\begin{align*}
t \mapsto \int_{\Om}u(t)\div(\phi)\dx
\end{align*}
is measurable with respect to the Lebesgue measure on $(0, T)$ for every test vector field $\phi$. On a metric measure space, one can explain a divergence operator for derivations $\mathfrak{d}$ and find an equivalent characterization of $\BV$ which relies on an integration by parts formula that is based on derivations and their divergence, see \cite{BuffaPhD, BuffaComiMiranda, DiMarino} and Section \ref{sec-setting} in this paper. Thus in the metric setting, we say that $u: (0, T) \to \BV(\Om)$ is weakly measurable if the mapping
\begin{align*}
t \mapsto \int_{\Om}u(t)\div(\mathfrak{d})\d\mu
\end{align*}
is measurable with respect to the Lebesgue measure on $(0, T)$ for every derivation $\mathfrak{d}$ in a certain class.

Once having a reasonable definition of the underlying function space at hand, our method of proof is aligned to the one proposed in the work of B\"ogelein, Duzaar and Marcellini \cite{boegeleintv}. For mappings $v: \Omstar \times (0, \infty) \to \R$ and time-independent Cauchy-Dirichlet data $u_0: \Omstar \to \R$ such that the condition $v = u_0$ is fulfilled on the parabolic boundary of $\Om \times (0, \infty)$ in the sense that $v(t) = u_0$ holds $\mu$-almost everywhere on $\Omstar \setminus \Om$ for almost every $t \in (0, T)$, we are going to consider the relaxed convex functionals 
\begin{align*}
\mathcal{F}_{\varepsilon}[v] := \int_0^T e^{-\frac{t}{\varepsilon}}\left[\frac{1}{2}\int_{\Omstar}|\partial_tv|^2\d\mu + \frac{1}{\varepsilon}\|Dv(t)\|(\Omstar)\right]\dt
\end{align*} 
for $\varepsilon \in (0, 1]$. The properties of the total variation allow the application of standard methods in the calculus of variations to ensure the existence of minimizers $u_{\varepsilon}$ of $\mathcal{F}_{\varepsilon}$. To prove the existence of the minimizers $u_{\varepsilon}$ of the relaxed convex functionals $\mathcal{F}_{\varepsilon}$ in our setting, we are going to apply a compactness result by Simon \cite{Simon}, see Lemma \ref{compactness} in Appendix \ref{app-compactness}. In particular, this is used to identify the limit of a sequence of functions in $L^1_w(0, T; \BV(\Omstar))$ since there are no standard compactness theorems that can be applied to this space. Now, these minimizers are expected to converge to parabolic minimizers as in (\ref{introquasi}) based on an idea in the Euclidean case. There, the authors compute the corresponding Euler-Lagrange-equation for $\mathcal{F}_{\varepsilon}$, see \cite{boegeleintv} for details. In order to establish this convergence, we are going to argue completely on the level of minimizers. To this end, we follow an idea of Lichnewsky \& Temam \cite{LichnewskyTemam} and thus introduce the concept of evolutionary variational solutions similarly to \cite{boegeleintv}. Precisely, in line with Definition \ref{varsoldef}), we are looking at continuous mappings $u: (0, T) \to L^2(\Omstar)$ that are also in the parabolic space $L^1_w(0, T; \BV_{u_0}(\Om))$ and fulfill the variational inequality 
\begin{align*}
\begin{aligned}
\int_0^T\|Du(t)\|(\Om^*)\dt & \leq \int_0^T\left[\int_{\Omstar}\partial_tv(v-u)\d\mu + \|Dv(t)\|(\Omstar)\right]\dt \\
& \mskip+25mu -\frac{1}{2}\|(v-u)(T)\|_{L^2(\Omstar)}^2 + \frac{1}{2}\|v(0) - u_0\|_{L^2(\Omstar)}^2.
\end{aligned}
\end{align*}
By introducing a mollification in time in order to establish the existence of an $L^2$-time derivative, energy estimates are shown which lead to the convergence of the minimizers $u_{\varepsilon}$ of $\mathcal{F}_{\varepsilon}$ to the variational solution $u$.
Finally, it can be shown that these variational solutions are actually parabolic minimizers.

\section{Setting and statement of results}\label{sec-setting}

\subsection{Notations}\label{subsec-notation}

Let $(\X, d, \mu)$ be a separable, connected metric measure space, i.e. $(\X, d)$ is a complete, separable and connected metric space endowed with a Borel measure $\mu$ on $\X$. The measure $\mu$ is assumed to fulfill a \textit{doubling property}, i.e. there exists a constant $c \geq 1$, such that
\begin{align}\label{double}
0 < \mu\left(B_{2r}(x)\right) \leq c\cdot \mu\left(B_r(x)\right) < \infty
\end{align}
for all radii $r > 0$ and centres $x \in \X$. Here $B_r(x) \coloneqq \{y \in \X: d(x, y) < r\}$ denotes the open ball with radius $r$ and centre $x$ with respect to the metric $d$. The \textit{doubling constant} is defined as
\begin{align}
c_d \coloneqq \inf\{c \geq 1: (\ref{double}) \text{ holds true}\}.
\end{align}
A complete metric measure space that fulfills the doubling property is proper, meaning that all closed and bounded subsets are compact, see \cite[Proposition 3.1]{Bjoern}. 

Following the concept of Heinonen and Koskela \cite{heinonenkoskela}, we call a Borel function $g: \X \to [0, \infty]$ an \textit{upper gradient} for an extended real-valued function $u: \X \to [-\infty, \infty]$ if for all $x, y \in \X$ and all rectifiable curves $\gamma: [0, L_{\gamma}] \to \X$ with $\gamma(0) = x, \gamma(L_{\gamma}) = y$ there holds 
\begin{align}\label{upgra}
|u(x) - u(y)| \leq \int_{\gamma}g\ds.
\end{align}
Note that we make use of the convention that $\infty - \infty = \infty$ and $(-\infty) - (-\infty) = -\infty$ as in \cite{Bjoern}. \\
Moreover, if a non-negative and measurable function $g$ fulfills (\ref{upgra}) for $p$-almost every curve as before, meaning that the family of curves for which (\ref{upgra}) fails has $p$-modulus zero, then $g$ is called $p$\textit{-weak upper gradient}. 
\\For $1 \leq p < \infty$ and a fixed open subset $\Omega \subset \X$ we define the vector space
\begin{align*}
\tilde{\N}^{1,p}(\Omega) \coloneqq \{u \in L^p(\Omega): \exists~p\text{-weak upper gradient } g \in L^p(\Omega) \text{ of } u\}.
\end{align*}
$L^p(\Omega)$ denotes the usual Lebesgue space. The space $\tilde{\N}^{1,p}(\Omega)$ is endowed with the semi-norm
\begin{align}\label{newt-norm}
\|u\|_{\tilde{\N}^{1,p}(\Omega)} \coloneqq \|u\|_{L^p(\Omega)} + \|g_u\|_{L^p(\Omega)},
\end{align}
where $g_u$ denotes the \textit{minimal $p$-weak upper gradient} of $u$, i.e. $\|g_u\|_{L^p(\Omega)} = \inf \|g\|_{L^p(\Omega)}$, with the infimum being taken over all $p$-weak upper gradients of $u$. Introducing the equivalence relation 
\begin{align*}
u \sim v \Longleftrightarrow \|u-v\|_{\tilde{\N}^{1,p}(\Omega)} = 0,
\end{align*}
we define the \textit{Newtonian space} $\N^{1,p}(\Omega)$ as the quotient space
\begin{align*}
\N^{1,p}(\Omega) \coloneqq \tilde{\N}^{1,p}(\Omega)/\sim,
\end{align*}
which we endow with the quotient norm $\|\cdot\|_{\mathcal{N}^{1,p}(\Omega)}$ defined as in \eqref{newt-norm}. Since this definition clearly depends on the metric $d$ and the measure $\mu$, we abuse the notation $\N^{1,p}(\Omega)$ as an abbreviation for $\N^{1,p}(\Omega, d, \mu)$. For more details on metric measure spaces we refer the reader to \cite{Bjoern,SobolevSpacesBook}. \\
In addition to the doubling property, we demand that the metric measure space $(\X, d, \mu)$ supports a weak $(1, 1)$-Poincar\'e inequality, in the sense that there exist a constant $c_P > 0$ and a dilatation factor $\tau \geq 1$ such that for all open balls $B_{\varrho}(x_0) \subset B_{\tau\varrho}(x_0) \subset \X$, for all $L^1$-functions $u$ on $\X$ and all upper gradients $\tilde{g}_u$ of $u$ there holds
\begin{align}\label{poincare}
\mint_{B_{\varrho}(x_0)}|u - u_{\varrho, x_0}|\d\mu \leq c_P\varrho\mint_{B_{\tau\varrho}(x_0)}\tilde{g}_u\d\mu,
\end{align}
where the symbol 
\begin{align*}
u_{\varrho, x_0} := \mint_{B_{\varrho}(x_o)}u\hspace{0.5mm}\d\mu := \frac{1}{\mu(B_{\varrho}(x_0))}\int_{B_{\varrho}(x_o)}u\hspace{0.5mm}\d\mu
\end{align*}
denotes the mean value integral of the function $u$ on the ball $B_{\varrho}(x_0)$ with respect to the measure $\mu$. We will omit the word 'weak' from here on and simply refer to this inequality as 'Poincar\'e inequality'. Poincar\'e inequalities on metric measure spaces have been studied quite extensively in the literature, see for example \cite{BjoernPoincare,BjoernPoincare2,heinonenkoskela2,KeithZhong,KinnunenCapacity,KosShanTuo,Laakso,saloffcoste,saloffcoste2}.\\

\begin{remark}
Throughout the paper, the doubling condition and the weak $(1, 1)$-Poincaré inequality are only used in the definition of derivations (the doubling condition implies that the space $\X$ is proper, hence bounded and closed sets are compact) and in the compactness result (see Appendix \ref{app-compactness}), which is exploited in the existence proof in Section \ref{sec-existence}. However, we decided to include both properties in the preliminary conditions of the underlying metric measure space as they play a paramount role in regularity theory, for which the presented existence proof in this manuscript lays the ground work.
\end{remark}

Now, we recall the definition and some basic properties of functions of bounded variation, see \cite{MirandaGoodSpaces}.
For $u \in L^1_{\text{loc}}(\X)$, we define the total variation of $u$ on $\X$ to be
\begin{align*}
\|Du\|(\X) \coloneqq\inf\left\{\liminf_{i\to \infty}\int_{\X}g_{u_i}\d\mu: u_i \in \Lip_{\text{loc}}(\X),~u_i \to u \text{ in } L^1_{\text{loc}}(\X)\right\},
\end{align*}
where each $g_{u_i}$ is the minimal $1$-weak upper gradient of $u_i$. We say that a function $u \in L^1(\X)$ is of \textit{bounded variation}, by notation $u \in \BV(\X)$, if $\|Du\|(\X) < \infty$. By replacing $\X$ with an open set $\Om \subset \X$ in the definition of the total variation, we can define $\|Du\|(\Om)$. The norm in $\BV$ is given by
\begin{align*}
\|u\|_{\BV(\Om)} \coloneqq \|u\|_{L^1(\Om)} + \|Du\|(\Om).
\end{align*}
It was shown in \cite[Theorem 3.4]{MirandaGoodSpaces} that for $u \in \BV(\X)$, the total variation $\|Du\|$ is the restriction to the class of open sets of a finite Radon measure defined on the class of all subsets of $\X$. This outer measure is obtained from the map $\Om \mapsto \|Du\|(\Om)$ on open sets $\Om \subset \X$ via the standard Carath\'eodory construction. Thus, for an arbitrary set $A \subset \X$ one can define
\begin{align*}
\|Du\|(A) \coloneqq \inf\left\{\|Du\|(\Om): \Om \text{ open}, A \subset \Om\right\}.
\end{align*}

\subsection{Parabolic function spaces}

For a Banach space $B$ and $T > 0$, the space
\begin{align*}
C^0([0, T]; B)
\end{align*}
consists of all continuous functions $u: [0, T] \to B$ with the norm
\begin{align*}
\|u\|_{C^0([0,T];B)} \coloneqq \max_{0 \leq t \leq T}\|u(t)\|_B.
\end{align*}
Naturally, for $\alpha \in (0, 1]$, the space
\begin{align*}
C^{0, \alpha}([0, T]; B)
\end{align*}
consists of those functions $u \in C^0([0, T]; B)$, for which additionally
\begin{align*}
\sup_{s, t \in [0, T]}\frac{\|u(s) - u(t)\|_B}{|s - t|^{\alpha}} < \infty
\end{align*}
holds true.

In the Euclidean case, it can be shown via integration by parts that the space $\BV$ can be written as the dual space of a separable Banach space, see \cite[Remark 3.12]{AmbFusPal}. Since this tool is not available in the metric setting (at least not in the sense as it is understood in the Euclidean case), a different approach has to be taken.

\subsubsection{The space $\BV$ via derivations}

For the following definitions and properties, we are going to follow \cite{DiMarino}, see also \cite{BuffaPhD,BuffaComiMiranda}. While in the literature derivations are explained for Lipschitz functions with bounded support, we write $\Lip_c$ instead since the underlying space $\X$ in this paper is proper.

Let 
$L^0(\X)$ denote the space of measurable functions on $\X$. By a (Lipschitz) \textit{derivation} we denote a linear map $\mathfrak{d}: \Lip_c(\X) \to L^0(\X)$ such that the Leibniz rule
\begin{align*}
\mathfrak{d}(fg) = f\mathfrak{d}(g) + g\mathfrak{d}(f)
\end{align*}
holds true for all $f, g \in \Lip_c(\X)$ and for which there exists a function $h \in L^0(\X)$ such that for $\mu$-a.e. (almost every) $x \in \X$ and all $f \in \Lip_c(\X)$ there holds
\begin{align}\label{DerRep}
|\mathfrak{d}(f)|(x) \leq h(x)\cdot\Lip_{a}(f)(x),
\end{align}
where $\Lip_{a}(f)(x)$ denotes the asymptotic Lipschitz constant of $f$ at $x$, i.e.
\begin{align*}
\Lip_a(f)(x) := \lim_{r \searrow 0} \sup_{y \in B_r(x)}\frac{|f(x) - f(y)|}{d(x, y)}.
\end{align*}
The set of all such derivations will be denoted by $\Der(\X)$. The smallest function $h$ satisfying (\ref{DerRep}) will by denoted by $|\mathfrak{d}|$ and we are going to write $\mathfrak{d} \in L^p$ when we mean to say $|\mathfrak{d}| \in L^p$.

For given $\mathfrak{d} \in \Der(\X)$ with $\mathfrak{d} \in L^1_{\mathrm{loc}}(\X)$ we define the \textit{divergence operator} $\div(\mathfrak{d}): \Lip_c(\X) \to \R$ as
\begin{align*}
f \mapsto -\int_{\X}\mathfrak{d}(f)\d\mu.
\end{align*}
We say that $\div(\mathfrak{d}) \in L^p(\X)$ if this operator admits an integral representation via a unique $L^p$-function $\tilde{h}$, i.e.
\begin{align*}
\int_{\X}\mathfrak{d}(f)\d\mu = - \int_{\X}\tilde{h}f\d\mu.
\end{align*}

For all $p, q \in [1, \infty]$ we shall set
\begin{align*}
\Der^p(\X) \coloneqq \{\mathfrak{d} \in \Der(\X): \mathfrak{d} \in L^p(\X)\}
\end{align*}
and 
\begin{align*}
\Der^{p, q}(\X) \coloneqq \{\mathfrak{d} \in \Der(\X): \mathfrak{d} \in L^p(\X),~\div(\mathfrak{d}) \in L^q(\X)\}.
\end{align*}
When $p=\infty=q$ we will write $\Der_{b}(\X)$ instead of $\Der^{\infty, \infty}(\X)$. The domain of the divergence is characterized as
\begin{align*}
D(\div) \coloneqq \{\mathfrak{d} \in \Der(\X): |\mathfrak{d}|, \div(\mathfrak{d}) \in L^1_{\loc}(\X)\}.
\end{align*}

For $u \in L^1(\X)$ we say that $u$ is of \textit{bounded variation (in the sense of derivations)} in $\X$, denoted $u \in \BV_{\mathfrak{d}}(\X)$, if there is a linear and continuous map $L_u: \Der_b(\X) \to \mathbf{M}(\X)$ such that
\begin{align}\label{BVpartial}
\int_{\X}\d L_u(\mathfrak{d}) = -\int_{\X}u\div(\mathfrak{d})\d\mu
\end{align}
for all $\mathfrak{d} \in \Der_b(\X)$ and satisfying $L_u(h\mathfrak{d}) = hL_u(\mathfrak{d})$ for any bounded $h \in \Lip(\X)$, where $\mathbf{M}(\X)$ denotes the space of finite signed Radon measures on $\X$.

As observed in \cite{DiMarino}, the characterization of $\BV$ in the sense of derivations is well-posed. If we take any two maps $L_u$, $\tilde{L}_u$ as in (\ref{BVpartial}), the Lipschitz-linearity of derivations ensures that $L_u(\mathfrak{d}) = \tilde{L}_u(\mathfrak{d})$ $\mu$-a.e. for all $\mathfrak{d} \in \Der_b(\X)$. The common value will be then denoted by $Du(\mathfrak{d})$.

From \cite{DiMarino} we know that for $u \in \BV_{\mathfrak{d}}(\X)$ there exists a non-negative, finite Radon measure $\nu \in \mathbf{M}(\X)$ such that for every Borel set $B \subset \X$ one has
\begin{align}\label{envelope}
\int_B\d Du(\mathfrak{d}) \leq \int_B|\mathfrak{d}|^{*}\d\nu
\end{align}
for all $\mathfrak{d} \in \Der_b(\X)$, where $|\mathfrak{d}|^{*}$ denotes the upper-semicontinuous envelope of $|\mathfrak{d}|$. The least measure $\nu$ satisfying (\ref{envelope}) will be denoted by $\|Du\|_{\mathfrak{d}}$, the total variation of $u$ (in the sense of derivations). Moreover, we have
\begin{align*}
\|Du\|_{\mathfrak{d}}(\X) = \sup\{|Du(\mathfrak{d})(\X)|: \mathfrak{d} \in \Der_b(\X), |\mathfrak{d}| \leq 1\}.
\end{align*}
Finally, by \cite[Theorem 7.3.4]{DiMarino}, the classical representation formula for $\|Du\|_{\mathfrak{d}}$ holds, in the sense that if $\Om \subset \X$ is any open set, then
\begin{align}\label{TVRep}
\|Du\|_{\mathfrak{d}}(\Om) = \sup\left\{\int_{\Om}u\div(\mathfrak{d})\d\mu: \mathfrak{d} \in \Der_b(\X), \supp(\mathfrak{d}) \Subset \Om, |\mathfrak{d}| \leq 1\right\}.
\end{align}
From \cite[Theorem 7.3.7]{DiMarino} we obtain that if $(\X, d, \mu)$ is a complete and separable metric measure space endowed with a locally finite measure $\mu$ (as in the case of this paper), then
\begin{align*}
\BV(\X) = \BV_{\mathfrak{d}}(\X)
\end{align*}
and in particular, the respective notions of the total variation coincide. Therefore, from now on, we are only going to write $\BV(\X)$ and $\|Du\|$ without making any further distinction.

From \cite{hk} and \cite{MirandaGoodSpaces} we take the following useful properties of the total variation:
\begin{lemma}\label{totvar-sublin}{\normalfont{\cite[Remark 3.2]{MirandaGoodSpaces}}}
Let $u, v \in L^1_{\loc}(\X)$. Then, for any open set $\Om \subset \X$ and $\alpha \in \R$ there holds:
\begin{itemize}
\item[i)]$\|D(\alpha u)\|(\Om) = |\alpha|\cdot \|Du\|(\Om)$.
\item[ii)]$\|D(u + v)\|(\Om) \leq \|Du\|(\Om) + \|Dv\|(\Om)$.
\end{itemize}
Combining i) and ii), we find that the mapping $u \mapsto \|Du\|(\Om)$ is convex.
\end{lemma}

\smallskip

\begin{proposition}\label{totvar-lsc}{\normalfont{\cite[Proposition 3.6]{MirandaGoodSpaces}}}
Let $\Om \subset \X$ be an open set and let $(u_n)_{n \in \NN}$ be a sequence in $\BV_{\loc}(\Om)$ such that $u_n \longrightarrow u$ in $L^1_{\loc}(\Om)$. Then, there holds
\begin{align*}
\|Du\|(A) \leq \liminf_{n \to \infty}\|Du_n\|(A)
\end{align*}
for any open set $A \subset \Om$. In particular, if $\sup\limits_{n \in \NN}\|Du_n\|(A) < \infty$ for any open set $A \Subset \Om$, the limit function $u$ is in $\BV_{\loc}(\Om)$.
\end{proposition}

\begin{lemma}\label{minmaxvar}{\normalfont{\cite[Theorem 2.8]{hk}}} Let $v,w\in \BV(\Omega)$. Then, $\min\{v,w\},\max\{v,w\}\in \BV(\Omega)$ and
\[
 \|D\min\{v,w\}\|(\Omega)+\|D\max\{v,w\}\|(\Omega)\le\|Dv\|(\Omega)+\|Dw\|(\Omega).
\]
\end{lemma}

\subsubsection{Weak parabolic function spaces}\label{sub-weak-par}

For $T > 0$ and an open subset $\Om \subset \X$ we write $\Om_T$  for the space-time cylinder $\Om \times (0, T)$. For the concept of variational solutions we are going to make use of the space 
\begin{align*}
L^1_w(0, T; \BV(\Om)),
\end{align*}
where the suffix \textit{w} stands for 'weak'. This space consists of those $v \in L^1(\Om_T)$, such that there holds:
\begin{itemize}
\item $v(\cdot, t) \in \BV(\Om)$ for a.e. $t \in (0, T)$,
\item The mapping $t \mapsto v(\cdot, t)$ is \textit{weakly measurable}, i.e. the mapping
\begin{align}\label{pairing}
(0, T) \ni t \longmapsto \int_{\Om}v(t)\div(\mathfrak{d})\d\mu 
\end{align}
is measurable for all $\mathfrak{d} \in \Der_b(\X)$ with $\supp(\mathfrak{d}) \Subset \Om$,
\item $\displaystyle\int_0^T\|Dv(t)\|(\Om)\dt < \infty$.
\end{itemize}

\begin{remark}
In the case of the gradient flow, i.e. a functional with $p$-growth for $p > 1$, the parabolic function spaces considered are usually $L^p(0, T; \N^{1,p}(\Om))$, which consist of mappings $v: (0, T) \to \N^{1,p}(\Om)$ that are strongly measurable in the sense of Bochner, see \cite[Chapter 3]{SobolevSpacesBook}. In the case at hand, that is $p = 1$, one would consider the Bochner space $L^1(0, T; \BV(\Om))$. But the strong measurability in the sense of Bochner is too restrictive, since many simple examples - like the space-time cone $u(t) = \mathds{1}_{B_t(x_0)}$ -  are not strongly measurable in the sense of Bochner, since their image is not separable in $\BV(\Om)$. Therefore, the strong measurability condition is replaced with a weaker one. \\
Note that the weak measurability of a function in $L^1_w(0, T; \BV(\Om))$ is not to be confused with the weak measurability of a Banach space-valued function in the sense of Pettis' theorem, see again  \cite[Chapter 3]{SobolevSpacesBook}.
\end{remark}

\begin{remark}
In the Euclidean case, i.e. $\X = \R^n$ for some $n \in \NN$, the notion of weak measurability as in (\ref{pairing}) is usually understood in the sense that the pairing
\begin{align*}
(0, T) \ni t \mapsto \langle Dv(t), \varphi \rangle = -\int_{\Om}v(t)\div(\varphi)\dx
\end{align*}
is measurable for any $\varphi \in C^1_0(\Om; \R^n)$.

Indeed, the approach by derivations as introduced before yields this classical notion of weak measurability. To understand this, define for any $\varphi \in C^1_0(\Om; \R^n)$ the mapping
\begin{align*}
\mathfrak{d}_{\varphi}: \Lip_{\text{bs}}(\Om) \ni f \mapsto \langle \varphi, Df\rangle.
\end{align*}
By Rademacher's theorem, the gradient $Df$ is defined almost everywhere on $\Om$ for a Lipschitz function $f$. It is easy to check that $\mathfrak{d}_{\varphi}$ fulfills the Leibniz rule and the property (\ref{DerRep}) with $g(x) = |\varphi(x)|$ almost everywhere. By integration by parts, we find that for the divergence operator of $\mathfrak{d}_{\varphi}$ there holds 
\begin{align*}
\div(\mathfrak{d}_{\varphi}): f \mapsto -\int_{\Om}\langle \varphi, Df\rangle\dx = \int_{\Om}\div(\varphi)f\dx.
\end{align*} 
Hence, the divergence of $\mathfrak{d}_{\varphi}$ is represented by $\div(\varphi)$. Thus, the weak measurability in the sense of (\ref{pairing}) yields the measurability of the mapping
\begin{align*}
(0, T) \ni t \mapsto \int_{\Om}v(t)\div(\varphi)\dx.
\end{align*}
\end{remark}

In view of (\ref{TVRep}), the mapping $[0, T] \ni t \mapsto \|Dv(t)\|(\Om)$ is measurable for $v \in L^1_w(0, T; \BV(\Om))$. 

Furthermore, the limit of a sequence of functions in $L^1_w(0, T; \BV({\Om}))$ with uniformly bounded total variation is again a $L^1_w(0, T; \BV(\Om))$-function:

\begin{lemma}\label{uinweak}
Suppose that the sequence $u_j \in L^1_w(0, T; \BV(\Om))$, $j \in \NN$, satisfies
\begin{align*}
\sup_{j \in \NN}\int_0^T\|Du_j(t)\|(\Om)\dt < \infty
\end{align*}
and $u_j \to u$ in $L^1(\Om_T)$ as $j \to \infty$. Then, $u \in L^1_w(0, T; \BV(\Om))$.
\end{lemma}
\begin{proof}
By Fubini's theorem we have, for a subsequence of $(u_j)_{j\in\mathbb{N}}$ - which we still label as $u_j$ - and for a.e. $t\in (0,T)$,
\begin{equation*}
u_{j}(t) \longrightarrow u(t)\ \textrm{in}\ L^{1}(\Omega).
\end{equation*}
Therefore, for any $\mathfrak{d}\in \Der_{b}(\Omega)$ there holds
\begin{align*}
\int_{\Omega}u(t)\div(\mathfrak{d})\d\mu &=\lim_{j\rightarrow\infty}\int_{\Omega}u_{j}(t)\div(\mathfrak{d})\d\mu.
\end{align*}
Since $u_{j}\in L^{1}_{w}(0,T;BV(\Omega))$, we then have that for any $\mathfrak{d}\in \Der_{b}(\Omega)$ the mapping
\begin{equation*}
(0,T)\ni t\mapsto\int_{\Omega}u(t)\div(\mathfrak{d})\d\mu 
\end{equation*}
is measurable.

By the lower semicontinuity of the total variation with respect to $L^{1}$-con\-ver\-gence on the time slices and Fatou's lemma, we conclude:
\begin{align*}
\int_{0}^{T}\Vert Du(t)\Vert(\Omega)\dt&\leq \int_{0}^{T}\liminf_{j\rightarrow\infty}\Vert Du_{j}(t)\Vert(\Omega)\dt\\
&\leq\liminf_{j\rightarrow\infty} \int_{0}^{T}\Vert Du_{j}(t)\Vert(\Omega)\dt<\infty.
\end{align*}
This implies $u(t)\in \BV(\Omega)$ for a.e. $t\in(0,T)$ and of course that $\|Du(t)\|\in L^1(0,T)$. In other words,
$u\in L^{1}_{w}(0,T; \BV(\Omega)$.

\end{proof}

\subsection{Variational solutions}

In the Euclidean case, i.e. $\X = \R^n$, one might consider the Cauchy-Dirichlet problem
\begin{align}\label{CauchyDirichlet}
\left\{\begin{array}{rl}
\partial_tu - \div\left(\dfrac{Du}{|Du|}\right) = 0 & \text{in } \Om_T, \\ 
u = u_0 & \text{on } \partial_{\text{par}}\Om_T,
\end{array}\right. 
\end{align}
where $\partial_{\text{par}}\Om_T := (\overline{\Om} \times \{0\}) \cup (\partial\Om \times (0, T))$ denotes the \textit{parabolic boundary} of $\Om_T$ and $u_0$ is some given boundary data. 

When trying to define a concept of Cauchy-Dirichlet problems like (\ref{CauchyDirichlet}) for $u \in L^1_w(0, T; \BV(\Om))$ on metric measure spaces, one has to overcome several difficulties. Indeed, similarly to what was already observed in \cite[Section 1.2]{boegeleintv}, we point out that also in our case boundary values of $\BV$-functions are delicate to manage, since the trace operator is not continuous with respect to the weak$^*$-convergence in $\BV(\Om)$ - see for instance \cite[Def. 3.11]{AmbFusPal} -  and the pairing in (\ref{pairing}). A suitable strategy to treat this issue is to consider a slightly larger domain $\Om^*$ that compactly countains the bounded open set $\Om$ and to assume that the datum $u_0$ is defined on $\Om^*$. The boundary condition $u = u_0$ on the lateral boundary $\partial\Om \times (0, T)$ could then be interpreted by requiring that $u(\cdot, t) = u_0$ a.e. on $\Om^* \setminus \Om$ for all $t \in (0, T)$. Thus said, from now on boundary values shall be understood in the following sense:

\begin{center}
given $u_0 \in \BV(\Om^*)$, a function $u$ belongs to $\BV_{u_0}(\Om)$\\if and only if $u \in \BV(\Om^*)$ and $u = u_0$ a.e. on $\Om^* \setminus \Om$.
\end{center}

The condition on the lateral boundary has to be read  in the sense that there holds $u(\cdot, t) \in \BV_{u_0}(\Om)$ for a.e. $t \in (0, T)$.

On the other hand, we do not have the possibility to explain derivatives such as in (\ref{CauchyDirichlet}). Therefore we cannot consider Cauchy-Dirichlet problems like this. However, by an idea of Lichnewsky and Temam (see \cite{LichnewskyTemam}), one can define the concept of \textit{variational solutions}. Since this concept for solutions to a Cauchy-Dirichlet problem is described purely on a variational level, it can be extended to the concept of metric measure spaces. 

To be precise, we assume $\Om$ to be open and bounded, $\Om^*$ open and bounded with $\Om \Subset \Om^*$ and 
\begin{align}\label{datum}
u_0 \in L^2(\Om^*) \cap \BV(\Om^*).
\end{align}
Where it makes sense, we are going to abbreviate $v(t) := v(\cdot, t)$.

\medskip

\begin{definition}\label{varsoldef}
Assume that the Cauchy-Dirichlet datum $u_0$ fulfills (\ref{datum}). A map $u: \Om_T^* \to \R$, $T \in (0, \infty)$ in the class
\begin{align*}
L^1_w\left(0, T; \BV_{u_0}(\Om)\right) \cap C^0\left([0, T]; L^2(\Om^*)\right)
\end{align*}
will be referred to as a \textit{variational solution} on $\Om_T$ to the Cauchy-Dirichlet problem for the total variation flow if and only if the variational inequality 
\begin{align}\label{varineq}
\begin{aligned}
\int_0^T\|Du(t)\|(\Om^*)\dt & \leq \int_0^T\left[\int_{\Omstar}\partial_tv(v-u)\d\mu + \|Dv(t)\|(\Omstar)\right]\dt \\
& \mskip+25mu -\frac{1}{2}\|(v-u)(T)\|_{L^2(\Omstar)}^2 + \frac{1}{2}\|v(0) - u_0\|_{L^2(\Omstar)}^2
\end{aligned}
\end{align}
holds true for any $v \in L^1_w\left(0, T; \BV_{u_0}(\Om)\right)$ with $\partial_tv \in L^2(\Omstar_T)$ and $v(0) \in L^2(\Omstar)$. A map $u: \Omstar_{\infty} \to \R$ is termed a \textit{global variational solution} if 
\begin{align*}
u \in L^1_w\left(0, T; \BV_{u_0}(\Om)\right) \cap C^0\left([0, T]; L^2(\Om^*)\right) \text{ for any } T > 0
\end{align*}
and $u$ is a variational solution on $\Om_T$ for any $T \in (0, \infty)$.
\end{definition}

Note that the time independent extension $v(\cdot, t) := u_0$ is an admissible comparison map in (\ref{varineq}). Therefore, we have that 
\begin{align*}
\int_0^T\|Du(t)\|(\Omstar)\dt < \infty
\end{align*}
for any variational solution $u$.

\subsection{Main results}

Our main results concern the existence, uniqueness and regularity of variational solutions as follows:

\begin{theorem}\label{maintheo1}
Suppose that the Cauchy-Dirichlet datum $u_0$ fulfills the requirements of (\ref{datum}). Then, there exists a unique global variation solution in the sense of Definition \ref{varsoldef}.
\end{theorem}

\begin{theorem}\label{maintheo2}
Suppose that the Cauchy-Dirichlet datum $u_0$ fulfills the requirements of (\ref{datum}). Then, any variational solution in the sense of Definition \ref{varsoldef} on $\Om_T$ with $T \in (0, \infty]$ satisfies
\begin{align*}
\partial_tu \in L^2(\Omstar) \text{ and } u \in C^{0, \frac{1}{2}}\left([0, \tau]; L^2(\Omstar)\right) \text{ for all } \tau \in \R \cap (0, T].
\end{align*}
Furthermore, for the time derivative $\partial_tu$ there holds the quantitative bound
\begin{align*}
\int_0^T\int_{\Omstar}|\partial_tu|^2\d\mu\dt \leq \|Du_0\|(\Omstar).
\end{align*}
Finally, for any $t_1, t_2 \in \R$ with $0 \leq t_1 < t_2 \leq T$ one has the energy estimate
\begin{align}\label{energymain}
\frac{1}{t_2 - t_1}\int_{t_1}^{t_2}\|Du(t)\|(\Omstar)\dt \leq \|Du_0\|(\Omstar).
\end{align}
\end{theorem}

\section{Preliminaries}

\subsection{Mollification in Time}
As variational solutions in the sense of Definition \ref{varsoldef} lack the appropriate time-regularity, they are in general not admissible as comparison maps in (\ref{varineq}). This is why a mollification procedure with respect to time (also known as \textit{time-smoothing}) has to be performed. Indeed, such a technique will make it possible to show that the time derivative of a variational solution exists and belongs to $L^{2}$.
Below, we shall present the precise construction of time-smoothing as illustrated in \cite{boegeleintv}.

\smallskip

Let $X$ be a Banach space and $v_{0}\in X$. Now, consider some $v\in L ^{r}(0,T;X)$ for some $1\leq r\leq\infty$ and define for $h\in (0,T]$ and $t\in [0,T]$ the mollification in time by 
\begin{equation}\label{mollification}
[v]_{h}^{v_{0}}(t):=e^{-\frac{t}{h}}v_{0}+\frac{1}{h}\int_{0}^{t}e^{-\frac{s-t}{h}}v(s)\ds.
\end{equation}
It can be shown that the mollified function $[v]_{h}^{v_{0}}$ solves the ordinary differential equation
\begin{equation}\label{dermollification}
\partial_{t}[v]_{h}^{v_{0}}=-\frac{1}{h}([v]_{h}^{v_{0}}-v)
\end{equation}
with initial condition $[v]_{h}^{v_{0}}(0)=v_{0}$.\\
Regarding the basic properties of the mollification $[\cdot]_{h}^{v_{0}}$, we refer to the following Lemma; see \cite[Appendix B]{Boegeleinpq} and \cite[Lemma 2.2]{KinnunenLindqvist} for the full proofs of the respective results.

\smallskip

\begin{lemma}\label{mlemma1} Suppose $X$ is a separable Banach space and $v_{0}\in X$. If $v\in L^{r}(0,T;X)$ for some $r\geq 1$, then the mollification $[v]_{h}^{v_{0}}$ defined in (\ref{mollification}) fulfills $[v]_{h}^{v_{0}}\in C^{\infty}([0,T];X)$ and for any $t_{0}\in (0,T]$ there holds
\begin{equation*}
\Vert [v]_{h}^{v_{0}}\Vert _{L^{r}(0,t_{0};X)}\leq\Vert v\Vert_{L^{r}(0,t_{0};X)}+\left[\frac{h}{r}\left(1-e^{-\frac{t_{0}r}{h}}\right)\right]^{\frac{1}{r}}\Vert v_{0}\Vert_{X}.
\end{equation*}
In the case $r=\infty$ the bracket $[\cdots]^{\frac{1}{r}}$ in the preceding inequality has to be interpreted as $1$. Moreover, $[v]_{h}^{v_{0}}\rightarrow v$ in $L^{r}(0,T;X)$ as $h\searrow 0$. Finally, if $v\in C^{0}([0,T];X)$, then $[v]_{h}^{v_{0}}\in C^{0}([0,T];X)$, $[v]_{h}^{v_{0}}(0)=v_{0}$, and moreover $[v]_{h}^{v_{0}}\rightarrow v$ in $C^{0}([0,T];X)$ as $h\searrow 0$.
\end{lemma}

With the next Lemma we show that $L^1_w(0,T;\BV(\Omega))$ is closed under time-smoothing:

\begin{lemma}\label{mlemma2} Let $T>0$, $v_{0}\in \BV(\Omega)$ and $v\in L ^{1}_{w}(0,T; \BV(\Omega))$. Then $[v]_{h}^{v_{0}}$ as defined in (\ref{mollification}) satisfies that $[v]_{h}^{v_{0}}\in  L ^{1}_{w}(0,T; \BV(\Omega))$. Moreover
\begin{equation*}
\Vert D [v]_{h}^{v_{0}}(t)\Vert(\Omega)\leq\left[\Vert Dv(t)\Vert(\Omega)\right]_{h}^{\|Dv_0\|(\Om)}\ \ \ \textrm{for any } t\in(0,T)
\end{equation*}
and
\begin{equation}\label{eqmollification}
\lim_{h\searrow 0}\int_{0}^{T}\Vert D [v]_{h}^{v_{0}}(t)\Vert(\Omega)\dt=\int_{0}^{T}\Vert Dv(t)\Vert(\Omega)\dt.
\end{equation}
\end{lemma}
\begin{proof}
By Lemma \ref{mlemma1} (applied with $X=L^{1}(\Omega)$ and $r=1$) we conclude that $[v]_{h}^{v_{0}}\rightarrow v$ in $L^{1}(0,T;L^{1}(\Omega))$, so that $[v]_{h}^{v_{0}}(t)\rightarrow v(t)$ in $L^{1}(\Omega)$ for almost every $t\in (0,T)$. As a test-map, let us consider a bounded derivation $\mathfrak{d}\in\Der_{b}(\Omega)$. Then, applying Fubini's Theorem and the definition of time-smoothing we obtain for any $t\in[0,T]$ that 
\begin{align*}
\begin{aligned}
&\int_{\Omega}[v]_{h}^{v_{0}}(t)\div(\mathfrak{d})\d\mu  = -e^{-\frac{t}{h}}\int_{\Omega}v_{0}\div(\mathfrak{d})\d\mu-\frac{1}{h}\int_{0}^{t}e^{\frac{s-t}{h}}\int_{\Omega}v(s)\div(\mathfrak{d})\d\mu\d s 
\end{aligned}
\end{align*}
holds true. This implies in particular that $$(0,T) \ni t\mapsto \displaystyle\int_{\Omega}[v]_{h}^{v_{0}}(t)\div(\mathfrak{d})\d\mu$$ is measurable. Moreover, taking the supremum over all $\mathfrak{d}\in\Der_{b}(X)$ with $\supp(\mathfrak{d})\Subset\Omega$ and $\vert\mathfrak{d}\vert\leq 1$, we conclude the following bound for the total variation of $[v]_{h}^{v_{0}}(t)$, i.e. we have that
\begin{align*}
\Vert D[v]_{h}^{v_{0}}(t)\Vert(\Omega) & \leq e^{-\frac{t}{h}}\Vert Dv_{0}\Vert(\Omega)+\frac{1}{h}\int_{0}^{t}e^{\frac{s-t}{h}}\Vert Dv(s)\Vert(\Omstar) \d s \\
& =\left[\Vert Dv(t)\Vert(\Omega)\right]_{h}^{\Vert Dv_{0}\Vert(\Omega)}<\infty
\end{align*}
holds true and therefore $[v]_{h}^{v_{0}}(t)\in \BV(\Omega)$.

Since $t\mapsto \Vert Dv(t)\Vert(\Omega)\in L ^{1}(0,T)$ by the definition of $L^{1}_{w}(0,T; \BV(\Omega))$ we obtain, using Lemma \ref{mlemma1} applied to $\Vert Du(t)\Vert(\Omega)$,
\begin{equation}\label{last-ineq}
\int_{0}^{T}\Vert D[v]_{h}^{v_{0}}(t)\Vert(\Omega)\d t\leq h \Vert Dv_{0}\Vert(\Omega)+\int_{0}^{T}\Vert Dv(t)\Vert (\Omega)\d t<\infty.
\end{equation}
This proves $[v]_{h}^{v_{0}}\in L^{1}_{w}(0,T; \BV(\Omega))$. Now, by the lower-semicontinuity of the total variation, a combination of Fatou's Lemma and \eqref{last-ineq} yields
\begin{align*}
\int_{0}^{T}\Vert Dv(t)\Vert (\Omega)\d t&\leq \int_{0}^{T}\liminf_{h\searrow 0}\Vert D[v]_{h}^{v_{0}}(t)\Vert(\Omega)\d t\\
&\leq\liminf_{h\searrow 0}\int_{0}^{T}\Vert D[v]_{h}^{v_{0}}(t)\Vert(\Omega)\d t\leq \int_{0}^{T}\Vert Dv(t)\Vert (\Omega)\d t,
\end{align*}
hence (\ref{eqmollification}) is established. The Lemma is thus proved.

\end{proof}

\subsection{Localizing the problem on a smaller cylinder}\label{subsec:local}
Let $u\in L_{w}^{1}(0,T; \BV_{u_{0}}(\Omega))\cap C^{0}([0,T]; L^{2}(\Omega^{*}))$ be a variational solution on some cylinder $\Omega_{T}^{*}$ with $T\in(0,\infty)$ in the sense of Definition \ref{varsoldef}.\\
In this section we want to prove that any such $u$ is a variational solution also on any subcylinder of the form $\Omega_{t_{1},t_{2}}^{*}\coloneqq\Omega^{*}\times(t_{1},t_{2})$ with $0\leq t_{1}<t_{2}\leq T$. To obtain such a localization, let $v\in L_{w}^{1}(t_{1},t_{2}; \BV_{u_{0}}(\Omega))$ with  $\partial_{t}v\in L^{2}(\Omega_{t_{1},t_{2}}^{*})$, $v(t_{1})\in L ^{2}(\Omega^{*})$, and choose for fixed $\vartheta\in \left(0,\frac{1}{2}(t_{2}-t_{1})\right)$ the cutoff function
\begin{align*}
\zeta_{\vartheta} := \left\{\begin{array}{cl}
0 & \text{ if } t \in [0, t_1], \\ 
\dfrac{1}{\vartheta}(t - t_1) & \text{ if } t \in (t_1, t_1 + \vartheta), \\ 
1 & \text{ if } t \in [t_1 + \vartheta, t_2 - \vartheta], \\ 
\dfrac{1}{\vartheta}(t_2 - t) & \text{ if } t \in (t_2 - \vartheta, t_2), \\ 
0 & \text{ if } t \in [t_2, T].
\end{array}\right. 
\end{align*}
The comparison map is now defined by $\tilde{v}:= \zeta_{\vartheta}v+(1-\zeta_{\vartheta})[u]_{h}^{u_{0}}$. Here we extended $\zeta_{\vartheta}$ outside of $\Omega^{*}\times [t_{1},t_{2}]$ by $0$.

Let us first check that $\tilde{v}$ is indeed admissible. By virtue of Lemma \ref{mlemma2} there holds  $[u]_{h}^{u_{0}}\in L_{w}^{1}(0,T; \BV_{u_{0}}(\Omega^{*}))$. Hence, $\zeta_{\vartheta}v, (1-\zeta_{\vartheta})[u]_{h}^{u_{0}}\in L_{w}^{1}(0,T; \BV_{u_{0}}(\Omega^{*})),$ and the same applies to $\tilde{v}$. Moreover, $\partial_{t}\tilde{v}\in L^{2}(\Omega_{T}^{*})$ since $\partial_{t}[u]_{h}^{u_{0}}\in L^{2}(\Omega_{T}^{*})$. Also, we have the validity of (\ref{dermollification}) and that $u\in C^{0}([0,T];L ^{2}(\Omega^{*}))$, so an application of Lemma \ref{mlemma1} entitles us to conclude that $[u]_{h}^{u_{0}}\in C^{0}([0,T];L ^{2}(\Omega^{*}))$. Finally, it can be easily seen that the other properties - like the boundary condition $\tilde{v}(x,t)=u_0(x)$ on $\Omstar$ are inherited from the corresponding properties of $v$ and $[v]_{h}^{v_{0}}$ and from the fact that $\tilde{v}$ is defined as a convex combination. Then, $\tilde{v}$ accounts as a suitable candidate for the variational inequality \eqref{varineq}.
There, we rewrite the integrand containing the time derivative in the following way: 
\begin{align*}
& \partial_{t}(\zeta_{\vartheta}v+(1-\zeta_{\vartheta})[u]_{h}^{u_{0}})(\zeta_{\vartheta}v+(1-\zeta_{\vartheta})[u]_{h}^{u_{0}}-u) \\
& \mskip+100mu =\zeta_{\vartheta}\partial_{t}v(\zeta_{\vartheta}v+(1-\zeta_{\vartheta})[u]_{h}^{u_{0}}-u+\zeta_{\vartheta}u-\zeta_{\vartheta}u)\\
& \mskip+150mu +(1-\zeta_{\vartheta})\partial_{t}[u]_{h}^{u_{0}}(\left[[u]_{h}^{u_{0}}-u\right]+\zeta_{\vartheta}(v-[u]_{h}^{u_{0}}))\\
& \mskip+150mu +\zeta'_{\vartheta}(v-[u]_{h}^{u_{0}})(\left[[u]_{h}^{u_{0}}-u\right]+\zeta_{\vartheta}(v-[u]_{h}^{u_{0}}))\\
& \mskip+100mu =\zeta_{\vartheta}^{2}\partial_{t}v(v-u)+\zeta_{\vartheta}(1-\zeta_{\vartheta})\partial_{t}v([u]_{h}^{u_{0}}-u)\\
& \mskip+150mu +(1-\zeta_{\vartheta})(\partial_{t}[u]_{h}^{u_{0}})([u]_{h}^{u_{0}}-u)\\
& \mskip+150mu +\zeta_{\vartheta}(1-\zeta_{\vartheta})(\partial_{t}[u]_{h}^{u_{0}})(v-[u]_{h}^{u_{0}})+\zeta_{\vartheta}'(v-[u]_{h}^{u_{0}})\\
& \mskip+150mu +\zeta_{\vartheta}'\zeta_{\vartheta}(v-[u]_{h}^{u_{0}})^{2}\\
& \mskip+100mu =\zeta_{\vartheta}^{2}\partial_{t}v(v-u)+(1-\zeta_{\vartheta})(\partial_{t}[u]_{h}^{u_{0}})([u]_{h}^{u_{0}}-u)\\
& \mskip+150mu +\zeta_{\vartheta}(1-\zeta_{\vartheta})\left[\partial_{t}v([u]_{h}^{u_{0}}-u)+\partial_{t}[u]_{h}^{u_{0}}(v-[u]_{h}^{u_{0}})\right]\\
& \mskip+150mu +\zeta_{\vartheta}'(v-[u]_{h}^{u_{0}})([u]_{h}^{u_{0}}-u)+\zeta_{\vartheta}'\zeta_{\vartheta}(v-[u]_{h}^{u_{0}})^{2}.
\end{align*}
In the integral containing $\Vert D\tilde{v}(t)\Vert(\Omega)$ we use the convexity of the total variation in the form
\begin{align}\label{convex-est-var}
\Vert D\tilde{v}(t)\Vert(\Omega)\leq\zeta_{\vartheta}(t)\Vert Dv(t)\Vert(\Omega)+(1-\zeta_{\vartheta}(t))\Vert D[u]_{h}^{u_{0}}(t)\Vert(\Omega).
\end{align}
Plugging the above estimate \eqref{convex-est-var} in the variational inequality, we find that the term 
\begin{align*}
\displaystyle\int_{0}^{T}\Vert Du(t)\Vert(\Omega)\d t
\end{align*}
is lesser or equal than
\begin{align}\label{intlocalizing}
\begin{aligned}
& \int_{t_{1}}^{t_{2}}\left[\int_{\Omega^{*}}\zeta_{\vartheta}^{2}(t)\partial_{t}v(v-u)\d\mu+\zeta_{\vartheta}(t)\Vert Dv(t)\Vert(\Omega^{*})\right]\d t  \\
& \mskip+20mu +\int_{0}^{T}(1-\zeta_{\vartheta}(t))\left[\int_{\Omega^{*}}\partial_{t}[u]_{h}^{u_{0}})([u]_{h}^{u_{0}}-u)\d\mu+\Vert D[u]_{h}^{u_{0}}(t)\Vert(\Omega^{*})\right]\d t  \\
& \mskip+20mu +\int_{t_{1}}^{t_{2}}\zeta_{\vartheta}(t)(1-\zeta_{\vartheta}(t))\int_{\Omega^{*}}\left[\partial_{t}v([u]_{h}^{u_{0}}-u)+\partial_{t}[u]_{h}^{u_{0}}(v-[u]_{h}^{u_{0}})\right]\d\mu\d t\\
& \mskip+20mu +\int_{t_{1}}^{t_{2}}\zeta_{\vartheta}'(t)\int_{\Omega^{*}}(v-[u]_{h}^{u_{0}})([u]_{h}^{u_{0}}-u)\d\mu\dt  \\
& \mskip+20mu +\int_{t_{1}}^{t_{2}}\zeta_{\vartheta}(t)\zeta_{\vartheta}'(t)\int_{\Omega^{*}}\vert v-[u]_{h}^{u_{0}}\vert^{2}\d\mu\d t  \\
& \mskip+20mu -\frac{1}{2}\Vert (\zeta_{\vartheta}v+(1-\zeta_{\vartheta})[u]_{h}^{u_{0}})(T)\Vert_{L^{2}(\Omega^{*})}^{2} +\frac{1}{2}\Vert (\zeta_{\vartheta}v+(1-\zeta_{\vartheta})[u]_{h}^{u_{0}})(0)\Vert_{L^{2}(\Omega^{*})}^{2}. 
\end{aligned}
\end{align}
Since $\zeta_{\vartheta}(T)=0$ and $\zeta_{\vartheta}(0)=0$ the two boundary terms in (\ref{intlocalizing}) simplify to
\begin{equation*}
-\frac{1}{2}\Vert ([u]_{h}^{u_{0}}-u)(T)\Vert_{L^{2}(\Omega^{*})}^{2}+\frac{1}{2}\Vert ([u]_{h}^{u_{0}}-u)(0)\Vert_{L^{2}(\Omega^{*})}^{2}.
\end{equation*}
Now, we pass to the limit as $\vartheta\searrow 0$ and arrive at 
\begin{align}\label{limit-theta-0}
\begin{aligned}
 \int_{t_{1}}^{t_{2}}\Vert Du(t)\Vert(\Omega)\d t & \leq \int_{t_{1}}^{t_{2}}\left[\int_{\Omega^{*}}\partial_{t}v(v-u)\d\mu+\Vert Dv(t)\Vert(\Omega^{*})\right]\d t\\
& \mskip+40mu +\frac{1}{2}\Vert (v-[u]_{h}^{u_{0}})(t_{1})\Vert_{L^{2}(\Omega^{*})}^{2}-\frac{1}{2}\Vert (v-[u]_{h}^{u_{0}})(t_{2})\Vert_{L^{2}(\Omega^{*})}^{2} \\
& \mskip+40mu +\int_{(0,t_{1})\cup(t_{2},T)}\bigg[\int_{\Omega^{*}}\partial_{t}[u]_{h}^{u_{0}}([u]_{h}^{u_{0}}-u)\d\mu \\ 
& \mskip+40mu +[\Vert D[u]_{h}^{u_{0}}(t)\Vert(\Omega^{*})-\Vert Du(t)\Vert(\Omega^{*})]\bigg]\d t  \\
& \mskip+40mu -\frac{1}{2}\Vert ([u]_{h}^{u_{0}}-u)(T)\Vert_{L^{2}(\Omega^{*})}^{2}+\frac{1}{2}\Vert ([u]_{h}^{u_{0}}-u)(0)\Vert_{L^{2}(\Omega^{*})}^{2}  \\
& \mskip+40mu +\int_{\Omega^{*}\times\lbrace t_{1}\rbrace}(v-[u]_{h}^{u_{0}})([u]_{h}^{u_{0}}-u)\d\mu  \\ 
& \mskip+40mu -\int_{\Omega^{*}\times\lbrace t_{2}\rbrace}(v-[u]_{h}^{u_{0}})([u]_{h}^{u_{0}}-u)\d\mu.
\end{aligned}
\end{align}
Above, we exploited the fact that the mixed term containing $\zeta_{\vartheta}(t)(1-\zeta_{\vartheta}(t))$, vanishes as $\vartheta\searrow 0$.\\
In the course of the proof we also used the identity 
\begin{align}\label{identity-loc}
\begin{aligned}
& \lim_{\vartheta\searrow 0}\int_{t_{1}}^{t_{2}}\zeta_{\vartheta}(t)\zeta_{\vartheta}'(t)\int_{\Omega^{*}}\vert v-[u]_{h}^{u_{0}}\vert^{2}\d\mu\d t \\ 
& \mskip+30mu = \frac{1}{2}\Vert (v-[u]_{h}^{u_{0}})(t_{1})\Vert_{L^{2}(\Omega^{*})}^{2} -\frac{1}{2}\Vert (v-[u]_{h}^{u_{0}})(t_{2})\Vert_{L^{2}(\Omega^{*})}^{2},\end{aligned}
\end{align}
which easily follows, since $v,[u]_{h}^{u_{0}}\in C^{0}([t_{1},t_{2}];L^{2}(\Omega^{*}))$ (Lemma \ref{mlemma1}).

By virtue of (\ref{dermollification}), $\partial_{t}[u]_{h}^{u_{0}}([u]_{h}^{u_{0}}-u)\leq 0$, so we can neglect the first term in the third line on the right-hand side of \eqref{limit-theta-0}.

To proceed, we now recall that $[u]_{h}^{u_{0}}\rightarrow u$ in $C^{0}([0,T];L^{2}(\Omega^{*}))$ and $\Vert D[u]_{h}^{u_{0}}\Vert(\Omega)\rightarrow\Vert Du\Vert(\Omega)$. These two facts, which both follow from Lemma \ref{mlemma1} and Lemma \ref{mlemma2}, eventually allow us to conclude that the remaining terms in the last two lines of \eqref{limit-theta-0} vanish in the limit as $h\searrow 0$.

Moreover, for the terms at the right-hand side of \eqref{identity-loc} we find $$\Vert (v-[u]_{h}^{u_{0}})(t_{i})\Vert_{L^{2}(\Omega^{*})}^{2}\rightarrow\Vert (v-u)(t_{i})\Vert_{L^{2}(\Omega^{*})}^{2}$$ for $i\in\lbrace 1,2\rbrace$. All in all, we get that $u$ is a variational solution on the smaller cylinder $\Omega_{t_{1},t_{2}}^{*}$, so it attains the variational inequality 
\begin{align*}
\int_{t_{1}}^{t_{2}}\Vert Du(t)\Vert(\Omega)\d t
&\leq \int_{t_{1}}^{t_{2}}\left[\int_{\Omega^{*}}\partial_{t}v(v-u)\d\mu+\Vert Dv(t)\Vert(\Omega^{*})\right]\d t\\
& \mskip+25mu +\frac{1}{2}\Vert (v-u)(t_{1})\Vert_{L^{2}(\Omega^{*})}^{2}-\frac{1}{2}\Vert (v-u)(t_{2})\Vert_{L^{2}(\Omega^{*})}^{2}
\end{align*}
for any $0\leq t_{1}<t_{2}\leq T$ and any test function satisfying $v\in L_{w}^{1}(t_{1},t_{2}; \BV_{u_{0}}(\Omega))$ with $\partial_{t}v\in L^{2}(\Omega_{t_{1},t_{2}}^{*})$.

\subsection{The initial condition}

In this brief section we shall see that variational solutions as in Definition \ref{varsoldef} satisfy the initial condition
$u(0)=u_{0}$ on $\Omega^{*}$ in the usual $L^{2}$-sense. In this respect, a key role is played by the time-growth of $\Vert u(t)-u_{0}\Vert_{L^{2}(\Omega^{*})}^{2}$, which is at most linear for $t>0$.

\smallskip

\begin{lemma}Assume that $u$ is a variational solution -  in the sense of Definition \ref{varsoldef} - on $\Omega^{*}_T$, $T\in(0,\infty]$. Then, the initial condition $u(0)=u_{0}$ is fulfilled in the usual $L^{2}$-sense, meaning that
\begin{equation}\label{initial-l2-sense}
\lim_{t\searrow 0}\Vert u(t)-u_{0}\Vert_{L^{2}(\Omega^{*})}^{2}=0.
\end{equation}
\end{lemma}
\begin{proof} From the discussion of Section \ref{subsec:local} we know that
$u$ satisfies the variational inequality in any subcylinder $\Omega_{\tau}^{*}$ for $\tau\in(0,T)$.\\
To prove \eqref{initial-l2-sense}, we rewrite the variational inequality (\ref{varineq}) on $\Omstar_{\tau}$ with the time independent extension of $u_{0}$ to $\Omstar_{\tau}$, namely $v(t)\equiv u_{0}$, for $t\in(0, \tau]$. Note that since $\Vert Du_{0}\Vert(\Omega)<\infty$, $v$ is admissible in (\ref{varineq}). For any $\tau\in(0,T)$, we then have
\begin{equation*}
\int_{0}^{\tau}\Vert Du(t)\Vert(\Omega)\d t+\frac{1}{2}\Vert u(\tau)-u_{0}\Vert_{L^{2}(\Omega^{*})}^{2}\leq\tau\Vert Du_{0}\Vert(\Omega)<\infty.
\end{equation*}
Thus, the assertion easily follows by discarding the non-negative energy term at the left-hand side and by letting $\tau\searrow 0$ at the right hand side in the above estimate.

\end{proof}

\subsection{The time derivative}

In this subsection, we are going to show Theorem \ref{maintheo2}. To this end, we are going to follow the strategy in \cite{boegeleintv}. 

\smallskip

Let $t_1$, $t_2 \in \R$ with $0 \leq t_1 < t_2 \leq T$ and recall that by Section \ref{subsec:local} $u$ satisfies the variational inequality on the subcylinder $\Omstar_{t_1, t_2}$. Set $\tilde{u}(s) \coloneqq u(s + t_1)$ for $s \in (0, t_2 - t_1)$ and observe that, if at the initial value $t_1$ we have $u(t_1) \in \BV(\Omstar)$, then $\tilde{u}$ satisfies the variational inequality (\ref{varineq}) on $\Omstar_{t_2 - t_1}$ with $u(t_1)$ instead of $u_0$. This is of course possible as $u(t_1)\in\BV(\Omstar)$ for a.e. $t_1$ and therefore in particular for $t_1=0$.
Also, Lemma \ref{mlemma1} ensures that $v = [\tilde{u}]_h^{u(t_1)}$ is admissible as a comparison map in the variational inequality for $\tilde{u}$ with initial data
\begin{equation*}
v(0) = \tilde{u}(0) = u(t_1).
\end{equation*}
Now, we test (\ref{varineq}) on $\Omstar_{t_2 - t_1}$ with $v = [\tilde{u}]_h^{u(t_1)}$ and discard the negative term $$-\dfrac{1}{2}\|([\tilde{u}]_h^{u(t_1)} - \tilde{u})(t_2 - t_1)\|_{L^2(\Omstar)}^2$$ on the right-hand side. Using Lemma \ref{mlemma2} and (\ref{dermollification}), we obtain that 
\begin{align*}
& -\int_0^{t_2 - t_1}\int_{\Omstar}\partial_t[\tilde{u}]_h^{u(t_1)}\left([\tilde{u}]_h^{u(t_1)} - \tilde{u}\right)\d\mu\dt \\
& \mskip+100mu \leq \int_0^{t_2 - t_1}\left[\|D[\tilde{u}]_h^{u(t_1)}(t)\|(\Omstar) - \|D\tilde{u}(t)\|(\Omstar)\right]\dt \\
& \mskip+100mu\leq \int_0^{t_2 - t_1}\left[\left[\|D\tilde{u}(t)\|(\Omstar)\right]_h^{\|Du(t_1)\|(\Omstar)} - \|D\tilde{u}(t)\|(\Omstar)\right]\dt \\
& \mskip+100mu= -h\int_0^{t_2 - t_1}\partial_t\left[\|D\tilde{u}(t)\|(\Omstar)\right]_h^{\|Du(t_1)\|(\Omstar)}\dt \\
& \mskip+100mu= h\left(\|D\tilde{u}(0)\|(\Omstar) - \left[\|D\tilde{u}(t)\|(\Omstar)\right]_h^{\|D\tilde{u}(0)\|(\Omstar)}(t_2 - t_1)\right).
\end{align*}
Dividing by $h$ and using again (\ref{dermollification}), the above inequality becomes
\begin{align}\label{est-divide-h}
\begin{aligned}
& \int_0^{t_2 - t_1}\int_{\Omstar}|\partial_t[\tilde{u}]_h^{u(t_1)}|^2\d\mu\dt \\
& \mskip+50mu \leq \|D\tilde{u}(0)\|(\Omstar) - \left[\|D\tilde{u}(t)\|(\Omstar)\right]_h^{\|D\tilde{u}(0)\|(\Omstar)}(t_2 - t_1) \leq \|D\tilde{u}(0)\|(\Omstar).
\end{aligned}
\end{align}
Notice that \eqref{est-divide-h} does not depend on $h>0$, whence we infer the existence of the time derivative $\partial_t\tilde{u} \in L^2(\Omstar_{\tau})$ together with the quantitative estimate
\begin{align*}
\int_0^{t_2 - t_1}\int_{\Omstar}|\partial_t\tilde{u}|^2\d\mu\dt \leq \|Du(t_1)\|(\Omstar).
\end{align*}
Now, recall that $\displaystyle[\|D\tilde{u}(t)\|(\Omstar)]_h^{\|D\tilde{u}(0)\|(\Omstar)}(\tau) \to \|D\tilde{u}(\tau)\|(\Omstar)$ as $h \searrow 0$ for a.e. $\tau \in (0, T)$, so \eqref{est-divide-h} eventually yields
\begin{align*}
\int_0^{t_2 - t_1}\int_{\Omstar}|\partial_t\tilde{u}|^2\d\mu\dt \leq \|D\tilde{u}(0)\|(\Omstar) - \|D\tilde{u}(t_2 - t_1)\|(\Omstar).
\end{align*}
Rewriting this last estimate with $u$ instead of $\tilde{u}$, we find that
\begin{align}\label{time1}
\int_{t_1}^{t_2}\int_{\Omstar}|\partial_tu|^2\d\mu\dt \leq \|D\tilde{u}(t_1)\|(\Omstar) - \|D\tilde{u}(t_2)\|(\Omstar)
\end{align}
holds true for a.e. $0 \leq t_1 < t_2 \leq T$. In particular, for $T<\infty$, (\ref{time1}) is still valid with $t_1 = 0$ and $t_2 = T$; if $T = \infty$ instead, we let $t_2 \to \infty$ to obtain the first assertion of the theorem. Moreover, for $t_1, t_2 \in \R$ with $0 \leq t_1 < t_2 \leq T$ we have
\begin{align}\label{l2-estimate}\begin{split}
\|u(t_2) - u(t_1)\|_{L^2(\Omstar)}^2 &= \int_{\Omstar}\left|\int_{t_1}^{t_2}\partial_tu\dt\right|^2\d\mu \\
& \leq |t_2 - t_1|\int_{t_1}^{t_2}\int_{\Omstar}|\partial_tu|^2\d\mu\dt \\
& \leq |t_2 - t_1| \cdot \|Du_0\|(\Omstar).\end{split}
\end{align}
If we now set $t_1 = 0$ in \eqref{l2-estimate}, we find for any $t \in \R \cap (0, T]$ that
\begin{align*}
\int_{\Omstar}|u(t)|^2\d\mu & \leq 2\int_{\Omstar}|u_0|^2\d\mu + 2\int_{\Omstar}|u(t) - u_0|^2\d\mu \\
& \leq 2\int_{\Omstar}|u_0|^2\d\mu + 2t\cdot\|Du_0\|(\Omstar).
\end{align*}
Therefore, we obtain
\begin{align*}
u \in C^{0, \frac{1}{2}}([0, \tau]; L^2(\Omstar)) \text{ for any } \tau \in \R \cap (0, T].
\end{align*}
To conclude the proof, it remains to establish the estimate (\ref{energymain}). Now that we know $\partial_tu$ to be in $L^2(\Omstar_T)$, we can apply integration by parts to rewrite the minimality condition (\ref{varineq}) in the form
\begin{equation*}
\int_0^{\tau}\|Du(t)\|(\Omstar)\dt \leq \int_0^{\tau}\left[\int_{\Omstar}\partial_tu(v-u)\d\mu + \|Dv(t)\|(\Omstar)\right]\dt
\end{equation*}
for any $\tau \in \R \cap (0, T]$. Now, for $t_1, t_2 \in \R$ with $0 \leq t_1 < t_2 \leq \tau$ we define
\begin{align*}
\zeta_{t_1, t_2}(t) := \left\{\begin{array}{cl}
1 & \text{ if } t \in [0, t_1], \\ 
\dfrac{t_2 - t}{t_2 - t_1} & \text{ if } t \in (t_1, t_2), \\ 
0 & \text{ if } t \in [t_2, \tau].
\end{array}\right.
\end{align*}
Let $v = u + \zeta_{t_1, t_2}([u]_h^{u_0} - u)$ be a comparison function in the minimality condition on $\Omstar_{\tau}$. First of all, $v$ is indeed admissible, meaning that $v \in L^1_w(0, \tau; \BV_{u_0}(\Omstar))$; moreover, $\partial_tv \in L^2(\Omstar_{\tau})$ by the arguments in Section \ref{subsec:local} and $v(0) = [u]_h^{u_0} \in L^2(\Omstar)$. Combining in addition the convexity of the total variation and Lemma \ref{mlemma2}, we find 
\begin{align*}
\int_0^{t_2}\|Du(t)\|(\Omstar)\dt & \leq \int_0^{t_2}\int_{\Omstar}\zeta_{t_1, t_2}\partial_tu([u]_h^{u_0} - u)\d\mu\dt \\
& \mskip+25mu + \int_0^{t_2}\left[(1 - \zeta_{t_1, t_2})\|Du(t)\|(\Omstar)+\zeta_{t_1, t_2}[\|Du(t)\|(\Omstar)]_h^{\|Du(0)\|(\Omstar)}\right]\dt.
\end{align*}
Now, after a suitable rearrangement of the terms, we use (\ref{dermollification}) and integrate by parts to find 
\begin{align*}
0 & \leq \int_0^{t_2}\int_{\Omstar}\zeta_{t_1, t_2}\partial_tu([u]_h^{u_0} - u)\d\mu\dt \\
& \mskip+25mu + \int_0^{t_2}\zeta_{t_1, t_2}\left([\|Du(t)\|(\Omstar)]_h^{\|Du(0)\|(\Omstar)} - \|Du(t)\|(\Omstar)\right)\dt \\
& = -h\int_0^{t_2}\int_{\Omstar}\zeta_{t_1, t_2}\partial_tu\partial_t[u]_h^{u_0}\d\mu\dt \\
& \mskip+25mu - h\int_0^{t_2}\zeta_{t_1, t_2}\partial_t[\|Du(t)\|(\Omstar)]_h^{\|Du(0)\|(\Omstar)}\dt \\
& = -h\int_0^{t_2}\int_{\Omstar}\zeta_{t_1, t_2}\partial_tu\partial_t[u]_h^{u_0}\d\mu\dt \\
& \mskip+25mu + h\int_0^{t_2}\zeta_{t_1, t_2}'[\|Du(t)\|(\Omstar)]_h^{\|Du(0)\|(\Omstar)}\dt + h\|Du_0\|(\Omstar).
\end{align*}
At this point, we divide both sides by $h > 0$ and pass to the limit as $h \searrow 0$, getting 
\begin{align*}
\frac{1}{t_2 - t_1}\int_{t_1}^{t_2}\|Du(t)\|(\Omstar)\dt & \leq \|Du_0\|(\Omstar) - \int_0^{t_2}\int_{\Omstar}\zeta_{t_1, t_2}|\partial_tu|^2\d\mu\dt  \leq \|Du_0\|(\Omstar).
\end{align*}
This concludes the proof of Theorem \ref{maintheo2}.

\section{Uniqueness}

We are now going to discuss a Comparison Principle for variational solutions of the total variation flow. 


\begin{lemma}[Comparison Principle]\label{comparelem}
Let $u,\tilde{u}$ be variational solutions on $\Omstar_T$ in the sense of Definition \ref{varsoldef}, $T\in(0,\infty]$, with Cauchy-Dirichlet data $u_0$ and $\tilde{u}_0\in L^2(\Omstar) \cap \BV(\Omega^*)$, respectively. Assume $u_0\le\tilde{u}_0$ $\mu$-almost everywhere on $\Omega^*$. Then, there holds $u\le\tilde{u}$ $(\mu\otimes\Leb^1)$-almost everywhere on $\Omega_T$.
\end{lemma}

\begin{proof}
 Let $\tau\in\mathbb{R}\cap(0,T]$. By the localization on smaller cylinders performed in Section \ref{subsec:local},
 we consider $v\coloneqq\min\{u,\tilde{u}\}$ and $w\coloneqq\max\{u,\tilde{u}\}$ as comparison maps in the variational inequalities for $u$ and $\tilde{u}$, respectively, on the smaller space-time domain $\Omega_\tau$.
 
 Of course, as $u,\tilde{u}\in L^1_w(0,T;\BV(\Omega^*))$, we have that $v=\min\{u,\tilde{u}\}\in L^1(\Omega^*_T)$. Then, by an application of Fubini's Theorem we find that for any $\mathfrak{d}\in\Der_b(\X)$ with $\supp(\frd)\Subset\Omega$, the mapping
 \[
  t\mapsto\int_\Omega v(t)\div(\frd)\d\mu 
 \]
is measurable with respect to $t$. Moreover, since $v\in \BV(\Omega)$ for almost every $t\in(0,T)$ – being the minimum of two $\BV$ functions – an application of Lemma \ref{minmaxvar} yields
\[
 \int_0^T \|Dv(t)\|(\Omstar)\d t<\infty,
\]
which means $v\in L^1_w(0,T;\BV(\Omega^*))$; by analogous arguments, one also shows that $w=\max\{u,\tilde{u}\}\in L^1_w(0,T;\BV(\Omega^*))$.

Now, since $u$ and $\tilde{u}$ are variational solutions, Theorem \ref{maintheo2} ensures that $\partial_t u,\partial_t\tilde{u}\in L^2(\Omega^*_T)$, hence $\partial_t v,\partial_t w\in L^2(\Omega^*_T)$ as well. Then, we can add the related variational inequalities (see (\ref{varineq})) to find
\begin{align}\label{comp-est}
\begin{aligned}
  \int_0^\tau \left[\|Du(t)\|(\Omstar)+\|D\tilde{u}(t)\|(\Omstar)\right]\d t &  \le \int_0^\tau\left[\|Dv(t)\|(\Omstar)+\|Dw(t)\|(\Omstar)\right]\d t \\ 
  & \mskip+25mu + \int_0^\tau\int_{\Omega^*}\left[\partial_t v(v-u)+\partial_t w(w-\tilde{u})\right]\d\mu\d t \\ 
  & \mskip+25mu -\frac{1}{2}\|(v-u)(\tau)\|^2_{L^2(\Omega^*)}-\frac{1}{2}\|(w-\tilde{u})(\tau)\|^2_{L^2(\Omega^*)} \\ 
  & \coloneqq \textbf{(A)} + \textbf{(B)} + \textbf{(C)} + \textbf{(D)},
\end{aligned}
\end{align}
where we employed the facts that $v(0)=u_0$ and $w(0)=\tilde{u}_0$. The meaning of the abbreviations $\textbf{(A)}$-$\textbf{(D)}$ is obvious. Thus said, we pass to estimate the terms \textbf{(A)}-\textbf{(D)} at the rightmost side of \eqref{comp-est}.

For the treatment of \textbf{(A)}, we find by Lemma \ref{minmaxvar} that there holds 
\[
\int_0^\tau\left[\|Dv(t)\|(\Omstar)+\|Dw(t)\|(\Omstar)\right]\d t\le \int_0^\tau\left[\|Du(t)\|(\Omstar)+\|D\tilde{u}(t)\|(\Omstar)\right]\d t.
\]

In order to estimate \textbf{(B)}, first notice that on the set $\left\{(x,t)\in\Omega_\tau^*:\;u(x,t)\le\tilde{u}(x,t)\right\}$ there holds
\begin{equation*}
 \partial_t v(v-u)+\partial_t w(w-\tilde{u})=0.
\end{equation*}
On the complementary set, namely $\left\{(x,t)\in\Omega_\tau^*:\;u(x,t)>\tilde{u}(x,t)\right\}$, one has instead
\begin{equation*}
 \partial_t v(v-u)+\partial_t w(w-\tilde{u})=\partial_t\tilde{u}(\tilde{u}-u)+\partial_t u(u-\tilde{u})=\frac{1}{2}\partial_t\left\vert u-\tilde{u}\right\vert^2. 
\end{equation*}
So, let us now combine these identities to find
\begin{align*}
 \int_0^\tau\int_{\Omega^*}\left[\partial_t v(v-u)+\partial_t w(w-\tilde{u})\right]\d\mu\d t & =\frac{1}{2}\int_0^\tau\int_{\Omega^*}\partial_t(u-\tilde{u})^2_{+}\d\mu\d t \\ & \le\frac{1}{2}\int_{\Omstar}(u-\tilde{u})^2_{+}(\tau)\d\mu.
\end{align*}

By the definition of $v$ and $w$ as the minimum and maximum of $u$ and $\tilde{u}$, respectively, we find that 
\begin{align*}
\frac{1}{2}\|(v - u)(\tau)\|_{L^2(\Omstar)}^2 = \frac{1}{2}\int_{\Omstar}(u - \tilde{u})_+^2(\tau)\d\mu = \frac{1}{2}\|(w - \tilde{u})(\tau)\|_{L^2(\Omstar)}^2,
\end{align*}
hence we can rewrite \textbf{(C)} and \textbf{(D)}.

All in all, by combining \textbf{(A)}-\textbf{(D)} above we find that
\begin{align*}
 \int_{\Omega^*\times\{\tau\}}(u-\tilde{u})^2_{+}\d\mu\le0
\end{align*}
holds true, so the arbitrariness of $\tau\in\R\cap(0,T]$ yields $u\le\tilde{u}$ $(\mu\otimes\Leb^1)$-almost everywhere in $\Om_T$.
 
\end{proof}

\section{The existence proof}\label{sec-existence}

In this section we provide the proof of Theorem \ref{maintheo1}. Recall that $\Omega\Subset\Omega^{*}$ are two bounded, open subsets of $\X$ and that $u_{0}\in \BV(\Omega^{*})\cap L^{2}(\Omega^{*})$ fulfills (\ref{datum}).

For the reader's convenience, we will omit the set $\Omstar$ in the total variation throughout this section whenever it is obvious, i.e. any term similar to $\|Dv\|$ has to be interpreted as the respective term similar to $\|Dv\|(\Omstar)$.

\subsection{A sequence of minimizers to a variational functional on $\Omega_{T}^{*}$}\label{subsec:sequence}

Let $T\in(0,\infty)$. For $\varepsilon\in(0,1]$ we shall concentrate on variational integrals of the form
\begin{equation}\label{f-eps-functional}
\mathcal{F}_{\varepsilon}(v)\coloneqq\int_{0}^{T}e^{-\frac{t}{\varepsilon}}\left[\frac{1}{2}\int_{\Omega^{*}}| \partial_{t}v|^{2}\d\mu+\frac{1}{\varepsilon}\|Dv(t)\|\right]\d t.
\end{equation}
Together with these functionals, in order to address the existence problem associated to \eqref{f-eps-functional} we are going to consider first a certain function space where the minimization shall be realized, namely
\begin{align*}
\mathcal{K}:=\left\{ v\in L_{w}^{1}(0,T;\BV(\Omega^{*})):\partial_{t}v\in L^{2}(\Omega^{*}_{T})\right\}.
\end{align*} 
On $\mathcal{K}$ we define the norm
\begin{align*}
\| v\|_{\mathcal{K}}\coloneqq\int_{0}^{T}\| v(t)\|_{\BV(\Omega^{*})}\d t+\|\partial_{t} v\|_{L^{2}(\Omega^{*}_{T})}.
\end{align*}
We note that there holds
\begin{align*}
e^{-\frac{T}{\varepsilon}}\| v\|_{\mathcal{K}}\leq\int_{0}^{T}e^{-\frac{t}{\varepsilon}}\| v(t)\|_{\BV(\Omega^{*})}\d t+\left[\int_{0}^{T}\int_{\Omega^{*}}| \partial_{t}v|^{2}\d\mu\d t\right]^{\frac{1}{2}}\leq \| v\|_{\mathcal{K}}.
\end{align*}
Moreover, we consider the subclass
\begin{equation*}
 \mathcal{K}_{u_0}\coloneqq \left\{v\in\mathcal{K};\; v=u_0\:\text{a.e.\:on}\:(\Omstar\setminus\Omega)\times(0,T)\:\text{and}\:v(0)=u_0\right\}.
\end{equation*}
Observe that since $$\|\partial_{t} v\|_{L^{2}(\Omega^{*})}<\infty\quad \text{and}\quad v(0)=u_{0}\in L^{2}(\Omega^{*}),$$ then $v\in C^{0,\frac{1}{2}}([0,T]; L^{2}(\Omega^{*}))$ and therefore the initial condition $v(0)=u_{0}$ is satisfied in the strong $L^{2}$-sense. Moreover, the time independent extension $u_{0}$ from $\BV(\Omega^{*})\cap L^{2}(\Omega^{*})$ to $\Omega^{*}_{T}$, namely $v(t)=u_{0}$ for $t\in(0,T]$ defines an element in $\mathcal{K}_{u_{0}}$.\\
More precisely, there holds
\begin{align*}
\| u\|_{\mathcal{K}}&=\int_{0}^{T}\| v(t)\|_{\BV(\Omega^{*})}\d t+\| \partial_{t} v\|_{L^{2}(\Omega^{*}_{T})}\\
&=\int_{0}^{T}\| u_{0}\|_{\BV(\Omega^{*})}\d t
+0\\
&=T\| u_{0}\|_{\BV(\Omega^{*})}.
\end{align*}

To continue, we introduce then the subclass of $\mathcal{K}_{u_{0}}$ consisting of functions for which the energy $\mathcal{F}_{\varepsilon}$ is finite, namely
\begin{align*}
\mathcal{K}^{\varepsilon}_{u_{0}}:=\left\{ u\in\mathcal{K}_{u_{0}}:\mathcal{F}_{\varepsilon}(u)<\infty\right\}.
\end{align*}
It is immediate to see that $\K_{u_0}^\eps$ is non-empty, as the time-independent extension of $u_0$ to the whole of
$\Omega^{*}_{T}$ is in $\K_{u_0}$ with finite energy \eqref{f-eps-functional}.
Actually, we explicitly point out that, since $\| Du_{0}\|(\Omega^{*})<\infty$, for such extension we precisely have
\begin{align*}
\mathcal{F}_{\varepsilon}(v)&=\int_{0}^{T}e^{-\frac{t}{\varepsilon}}\frac{1}{\varepsilon}\|Dv(t)\|\d t+\int_{0}^{T}e^{-\frac{t}{\varepsilon}}\left[\frac{1}{2}\int_{\Omega^{*}}| \partial_{t}v|^{2}\right]\d t\\
&=\int_{0}^{T}e^{-\frac{t}{\varepsilon}}\frac{1}{\varepsilon}\|Du_{0}\|\d t\\
&=(1-e^{-\frac{T}{\varepsilon}})\| Du_{0}\|\leq\| Du_{0}\|<\infty.
\end{align*}
It is worth mentioning that the condition $\partial_{t}v\in L^{2}(\Omega^{*}_{T})$ makes it possible to establish an $L^{1}$-bound for $v$ in terms of $\partial_{t}v$ and $u_{0}$. Indeed, given $v\in \mathcal{K}^{\varepsilon}_{u_{0}}$ for any $t\in[0,T]$ we infer that
\begin{align*}
\| v(t)\|_{L^{1}(\Omega^{*})}&\leq \| v(t)-u_{0}\|_{L^{1}(\Omega^{*})}+\| u_{0}\|_{L^{1}(\Omega^{*})}\\
&=\int_{\Omega^{*}}\bigg|\int_{0}^{t}\partial_{\tau}v(\tau)\d \tau\bigg|\d\mu+\| u_{0}\|_{L^{1}(\Omega^{*})}\\
&\leq\int_{\Omega^{*}_{T}}|\partial_{\tau}v(\tau)|\d \tau\d\mu+\| u_{0}\|_{L^{1}(\Omega^{*})}\\
&\leq(T\mu(\Om))^{\frac{1}{2}}\| \partial_{t} v \|_{L^{2}(\Omega^{*}_{T})}+\| u_{0}\|_{L^{1}(\Omega^{*})}.
\end{align*}
As we integrate this inequality with respect to $t\in(0,T)$ we finally get
\begin{align}\label{ineqv}
\| v\|_{L^{1}(\Omega^{*}_{T})}\leq T\left[(T\mu(\Om))^{\frac{1}{2}}\| \partial_{t} v \|_{L^{2}(\Omega^{*}_{T})}+\| u_{0}\|_{L^{1}(\Omega^{*})}\right].
\end{align}

\smallskip

Let us now prove the existence of minimizers for $\mathcal{F}_{\varepsilon}$:

\begin{lemma}\label{exlem}
For any $\varepsilon\in (0,1]$, there is a unique minimizer $u_\eps$ for the variational energy $\mathcal{F}_{\varepsilon}$ in the subclass $\mathcal{K}^{\varepsilon}_{u_{0}}$.
\end{lemma}
\begin{proof}
Let us apply (\ref{ineqv}) to $v\in \mathcal{K}^{\varepsilon}_{u_{0}}$. We find
\begin{align*}
\| v\|_{\mathcal{K}}&=\int_{0}^{T}\| v(t)\|_{\BV(\Omega^{*})}\d t+\| \partial_{t} v\|_{L^{2}(\Omega^{*}_{T})}\\
&=\int_{0}^{T}\| v(t)\|_{L^{1}(\Omega^{*})}\d t+\int_{0}^{T}\|  Dv(t)\|\d t+\| \partial_{t} v\|_{L^{2}(\Omega^{*}_{T})}\\
&\leq T\left[(T\mu(\Om))^{\frac{1}{2}}\| \partial_{t} v \|_{L^{2}(\Omega^{*}_{T})}+\| u_{0}\|_{L^{1}(\Omega^{*})}\right]+\int_{0}^{T}\|  Dv(t)\|\d t+\| \partial_{t} v\|_{L^{2}(\Omega^{*}_{T})}\\
&=\int_{0}^{T}\|  Dv(t)\|\d t+\left(1+T(T\mu(\Om))^{\frac{1}{2}}\right)\| \partial_{t} v \|_{L^{2}(\Omega^{*}_{T})}+T\| u_{0}\|_{L^{1}(\Omega^{*})}\\
&\leq\left(1+T(T\mu(\Om))^{\frac{1}{2}}\right)\left[\int_{0}^{T}\|Dv(t)\|\d t+\| \partial_{t} v \|_{L^{2}(\Omega^{*}_{T})}^{2}+1\right]+T\| u_{0}\|_{L^{1}(\Omega^{*})}.
\end{align*}
Here, we made use of the elementary inequality $a \leq a^2 + 1$ for all $a \geq 0$.

Now, since
\begin{enumerate}
\item[i)] $\displaystyle\dfrac{2e^{\frac{T}{\varepsilon}}}{\varepsilon}\geq 1\Rightarrow\dfrac{2e^{\frac{T}{\varepsilon}}}{\varepsilon}\int_{0}^{T}\|  Dv(t)\|\d t\geq \int_{0}^{T}\|Dv(t)\|\d t$,
\item[ii)] $2e^{\frac{T}{\varepsilon}}\geq 1$,
\item[iii)] $\displaystyle2e^{\frac{T}{\varepsilon}}\int_{0}^{T}e^{-\frac{t}{\varepsilon}}\left[\frac{1}{2}\int_{\Omega^{*}}|\partial_{t}v|^{2}\d\mu\right]\d t$  $\displaystyle\geq e^{\frac{T}{\varepsilon}}\int_{0}^{T}\int_{\Omega^{*}}e^{-\frac{t}{\varepsilon}}|\partial_{t}v|^{2}\d\mu\d t=\| \partial_{t} v \|_{L^{2}(\Omega^{*}_{T})}^{2}$,
\end{enumerate}
there follows 
\begin{align*}
& \left(1+T^{\frac{3}{2}}\mu(\Om)^{\frac{1}{2}}\right)\left[\int_{0}^{T}\|  Dv(t)\|\d t+\| \partial_{t} v \|_{L^{2}(\Omega^{*}_{T})}^{2}+1\right] \\
& \mskip+50mu\leq\left(1+T^{\frac{3}{2}}\mu(\Om)^{\frac{1}{2}}\right)2e^{\frac{T}{\varepsilon}}\left[\frac{1}{\varepsilon}\int_{0}^{T}\|  Dv(t)\|\d t + \int_{0}^{T}e^{-\frac{t}{\varepsilon}}\left[\frac{1}{2}\int_{\Omega^{*}}|\partial_{t}v|^{2}\d\mu\right]\d t+1\right]\\
& \mskip+50mu =2e^{\frac{T}{\varepsilon}}\left(1+T(T\mu(\Om))^{\frac{1}{2}}\right)\left[\mathcal{F}_{\varepsilon}(v)+1\right]
\end{align*}
and thus 
\begin{align*}
& \left(1+T^{\frac{3}{2}}\mu(\Om)^{\frac{1}{2}}\right)\left[\int_{0}^{T}\|  Dv(t)\|\d t+\| \partial_{t} v \|_{L^{2}(\Omega^{*}_{T})}^{2}+1\right]+T\| u_{0}\|_{L^{1}(\Omega^{*})} \\ 
& \mskip+125mu\leq 2e^{\frac{T}{\varepsilon}}\left(1+T(T\mu(\Om))^{\frac{1}{2}}\right)\left[\mathcal{F}_{\varepsilon}(v)+1\right]+T\| u_{0}\|_{L^{1}(\Omega^{*})}.
\end{align*}
Therefore, we obtain
\begin{equation}\label{norm-K-est}
\| u\|_{\mathcal{K}}\leq 2e^{\frac{T}{\varepsilon}}\left(1+T(T\mu(\Om))^{\frac{1}{2}}\right)\left[\mathcal{F}_{\varepsilon}(v)+1\right]+T\| u_{0}\|_{L^{1}(\Omega^{*})}.
\end{equation}
Now, we consider a minimizing sequence $(u_{j})\in\mathcal{K}^{\varepsilon}_{u_{0}}$, $j\in\mathbb{N}$, i.e.
\begin{align*}
\lim_{j\to\infty}\mathcal{F}_{\varepsilon}(u_{j})=\inf_{u\in\mathcal{K}^{\varepsilon}_{u_{0}}}\mathcal{F}_{\varepsilon}(u)\leq \mathcal{F}_{\varepsilon}(u_{0})=\| Du_{0}\|(1-e^{-\frac{T}{\varepsilon}}).
\end{align*}
To simplify the notation here we have called $u_{0}$ the time independent extension of $u_{0}$ to $\Omega_{T}^{*}$, which we already know belongs to $\mathcal{K}^{\varepsilon}_{u_{0}}$.

Now, suppose that the minimizing sequence satisfies - without loss of generality - $\mathcal{F}_{\varepsilon}(u_{j})\leq\| Du_{0}\|$. Then, \eqref{norm-K-est} yields
\begin{align*}
\| u_{j}\|_{\mathcal{K}}&\leq 2e^{\frac{T}{\varepsilon}}\left(1+T(T\mu(\Om))^{\frac{1}{2}}\right)\left[\mathcal{F}_{\varepsilon}(u_{j})+1\right]+T\| u_{0}\|_{L^{1}(\Omega^{*})}\\
&\leq 4e^{\frac{T}{\varepsilon}}\left(1+T(T\mu(\Om))^{\frac{1}{2}}\right)\left[\| Du_{0}\|+1\right]+T\| u_{0}\|_{L^{1}(\Omega^{*})},
\end{align*}
meaning that $(u_{j})_{j\in\mathbb{N}}$ is uniformly bounded with respect to $\|\cdot\|_{\mathcal{K}}$. From this we get the following estimate 
\begin{align*}
&\sup_{j \in \NN}\left[\int_{0}^{T}\| u_{j}(t)\|_{\BV(\Omega^{*})}\d t+\|\partial_{t} u_{j}\|_{L^{2}(\Omega^{*}_{T})}\right]\\
& \mskip+50mu \leq 4e^{\frac{T}{\varepsilon}}\left(1+T(T\mu(\Om))^{\frac{1}{2}}\right)\left[\| Du_{0}\|+1\right]+T\| u_{0}\|_{L^{1}(\Omega^{*})}
\end{align*}
to hold uniformly.

At this point, we invoke a compactness result by Simon \cite[Theorem 1]{Simon} with $p=1$ and $B=L^{1}(\Omega^{*})$ to infer that $(u_{j})_{j\in\mathbb{N}}$ is relatively compact in $L^{1}(\Omega^{*}_{T})$; see Lemma  \ref{compactness} for the exact statement. This allows to infer the existence of a subsequence - still denoted by $(u_{j})_{j\in\mathbb{N}}$ -  and of a measurable function $u:\Omega^{*}_{T}\rightarrow\mathbb{R}$, such that
\begin{align*}
\left\{\begin{array}{cl}
u_j \longrightarrow u & \text{strongly in } L^1(\Omstar_T), \\ 
u_j \longrightarrow u & \text{a.e. in } \Omstar_T, \\ 
u_j \longweak u & \text{weakly in } L^2(\Omstar_T), \\ 
\partial_tu_j \longweak \partial_tu & \text{weakly in } L^2(\Omstar_T).
\end{array}\right.
\end{align*}
Then, an application of Lemma \ref{uinweak} to $(u_{j})_{j\in\mathbb{N}}$ yields that $u\in L_{w}^{1}(0,T;\BV(\Omega^{*}))$. We now continue by exploiting the pointwise a.e. convergence to find
\begin{align*}
| u(x,t)-u_{0}(x)|\leq| u(x,t)-u_{j}(x,t)|+| u_{j}(x,t)-u_{0}(x)|\longrightarrow 0\ \ \ \ \textrm{as }j\rightarrow \infty
\end{align*}
for a.e. $(x,t)\in(\Omega^{*}\setminus\Omega)\times(0,T)$. Here we used the fact that the second term on the right-hand side is identically zero. So, $u(x,t)=u_{0}(x)$ for a.e. $(x,t)\in(\Omega^{*}\setminus\Omega)\times(0,T)$. Next, we observe that for any $0\leq s<t\leq T$ there holds
\begin{align*}
| u_{j}(t)-u_{j}(s)|&=\bigg|\int_{s}^{t}\partial_{\tau}u_{j}(\tau)\d\tau\bigg|\leq\int_{s}^{t}|\partial_{\tau}u_{j}(\tau)|\d\tau\\
&\leq |t-s|^{\frac{1}{2}}\left(\int_{s}^{t}|\partial_{\tau}u_{j}(\tau)|^{2}\d\tau\right)^{\frac{1}{2}}\\
&\leq| t-s|^{\frac{1}{2}}\left(\int_{0}^{T}|\partial_{\tau}u_{j}(\tau)|^{2}\d\tau\right)^{\frac{1}{2}}.
\end{align*}
By squaring both sides and integrating over $\Omstar$ with respect to the measure $\mu$, we arrive at
\begin{align*}
\| u_{j}(t)-u_{j}(s)\|_{L^{2}(\Omega^{*})}\leq\sqrt{| t-s|}\|\partial_{t}u_{j}\|_{L^{2}(\Omega_{T}^{*})}.
\end{align*}
Now, since there holds
\begin{align*}
\|\partial_{t}u_{j}\|_{L^{2}(\Omega_{T}^{*})}&\leq\|\partial_{t}u_{j}\|_{L^{2}(\Omega_{T}^{*})}^{2}+1\\
&\leq 2e^{\frac{T}{\varepsilon}}\left[\mathcal{F}_{\varepsilon}(u_{j})+1\right]\\
&\leq 2e^{\frac{T}{\varepsilon}}\left[\| Du_{0}\|+1\right],
\end{align*}
we obtain
\begin{align}\label{sqrt-estimate}\begin{split}
\| u_{j}(t)-u_{j}(s)\|_{L^{2}(\Omega^{*})}&\leq\sqrt{| t-s|}\|\partial_{t}u_{j}\|_{L^{2}(\Omega_{T}^{*})} \\
&\leq2e^{\frac{T}{\varepsilon}}\left[\| Du_{0}\|+1\right]\sqrt{| t-s|}.\end{split}
\end{align}
By the weak convergence $u_{j} \weak u$ in $L^{2}(\Omega_{T}^{*})$ and the fact that $u_{j}(0)=u_{0}$, rewriting \eqref{sqrt-estimate} with $s=0$ gives
\begin{align*}
\frac{1}{h}\int_{0}^{h}\| u(t)-u_{0}\|^{2}_{L^{2}(\Omega^{*})}\d t&\leq\liminf_{j\rightarrow\infty}\frac{1}{h}\int_{0}^{h}\| u_{j}(t)-u_{0}\|^{2}_{L^{2}(\Omega^{*})}\d t\\
&\leq he^{\frac{T}{\varepsilon}}\left[\| Du_{0}\|+1\right].
\end{align*}
In other words, 
\begin{align*}
\lim\limits_{h\searrow 0}\displaystyle\frac{1}{h}\int_{0}^{h}\| u(t)-u_{0}\|_{L^{2}(\Omega^{*})}^2\d t=0,
\end{align*}
which is the same as $u(0)=u_{0}$ in the usual $L^2$-sense; therefore, the initial condition is preserved in the limit. Let us now exploit the lower-semicontinuity of the total variation with respect to the $L^{1}$-convergence
and of the $L^{2}$-norm with respect to weak convergence; combining this with Fatou's Lemma yields 
\begin{align*}
&\int_{0}^{T}e^{-\frac{t}{\varepsilon}}\left[\int_{\Omega^{*}}\frac{1}{2}|\partial_{t}u|^{2}\d\mu+\frac{1}{\varepsilon}\| Du(t)\|\right]\d  t\\
& \mskip+150mu \leq \liminf_{j\rightarrow\infty}\int_{0}^{T}\int_{\Omega^{*}}e^{-\frac{t}{\varepsilon}}\frac{1}{2}|\partial_{t} u_{j}|^{2}\d\mu\d  t\\
& \mskip+170mu +\int_{0}^{T}e^{-\frac{t}{\varepsilon}}\frac{1}{\varepsilon}\left[\liminf_{j\rightarrow\infty}\| Du_{j}(t)\|\right]\d  t\\
& \mskip+150mu \leq\liminf_{j\rightarrow\infty}\int_{0}^{T}\int_{\Omega^{*}}e^{-\frac{t}{\varepsilon}}\frac{1}{2}|\partial_{t} u_{j}|^{2}\d\mu\d  t \\
& \mskip+170mu + \liminf_{j\rightarrow\infty}\int_{0}^{T}e^{-\frac{t}{\varepsilon}}\frac{1}{\varepsilon}\| Du_{j}(t)\|\d  t\\
& \mskip+150mu \leq\liminf_{j\rightarrow\infty}\int_{0}^{T}e^{-\frac{t}{\varepsilon}}\left[\int_{\Omega^{*}}\frac{1}{2}|\partial_{t} u_{j}|^{2}\d\mu+\frac{1}{\varepsilon}\| Du_{j}(t)\|\right]\d  t\\
& \mskip+150mu =\lim_{j\rightarrow\infty}\mathcal{F}_{\varepsilon}(u_{j})
\end{align*}
In other words, $u\in\K_{u_0}^{\eps}$ minimizes $\mathcal{F}_{\varepsilon}$. Lastly, uniqueness is a consequence of the fact that the time-derivative term in $\mathcal{F}_{\varepsilon}$ entails its strict convexity.

\end{proof}

\subsection{Rewriting the minimality condition}\label{subsec:rewrite}

Let $\varepsilon \in (0, 1]$ and denote by $u_{\varepsilon} \in \K_{u_0}^{\varepsilon}$ the unique minimizer of the variational energy $\mathcal{F}_{\varepsilon}$ in $\K_{u_0}^{\varepsilon}$. Assume $\varphi \in L^1_w(0, T; \BV_0(\Om))$ is such that $\partial_t\varphi \in L^2(\Omstar_T)$, $\varphi(0) \in L^2(\Omstar)$, and
\begin{align}\label{phifinite}
\int_0^T\|D(u_{\varepsilon} + \varphi)(t)\|\dt < \infty.
\end{align}
Given a Lipschitz function $\zeta:(0,T)\to[0,1]$, $\delta \in (0, e^{-\frac{T}{\varepsilon}}]$ and $(x, t) \in \Omstar_T$ we set
\begin{align*}
\sigma(t) \coloneqq \delta e^{\frac{t}{\varepsilon}}\zeta(t) 
\end{align*}
and
\begin{align*}
v_{\varepsilon, \delta}(x, t) \coloneqq u_{\varepsilon}(x, t) + \sigma(t)\varphi(x, t) = u_{\varepsilon}(x, t) + \delta e^{\frac{t}{\varepsilon}}\zeta(t)\varphi(x, t).
\end{align*}
We assume that either $\zeta(0) = 0$ or $\varphi(0) = 0$. Rearranging terms, we find that 
\begin{align}\label{sigmaconvcomb}
v_{\varepsilon, \delta}(x, t) = (1 - \sigma(t))u_{\varepsilon}(x, t) + \sigma(t)(u_{\varepsilon}(x, t) + \varphi(x, t)).
\end{align}
Thus said, by analogous localization arguments as in Section \ref{subsec:local} we can claim that  $v_{\varepsilon, \delta} \in$ $L^1_w(0, T; \BV_{u_0}(\Om))$. Moreover, the bound $\mathcal{F}_{\eps}(v_{\eps,\delta})<\infty$ is a consequence of the convexity of the total variation, since by (\ref{sigmaconvcomb})(note that $0 \leq \sigma(t) \leq 1$), $v_{\varepsilon, \delta}$ is a convex combination of $u_{\varepsilon}$ and $u_{\varepsilon} + \varphi$ on fixed time slices $t \in [0, T]$. More precisely, there holds the estimate
\begin{align*}
\int_0^Te^{-\frac{t}{\varepsilon}}\|Dv_{\varepsilon, \delta}(t)\|\dt &\leq \int_0^Te^{-\frac{t}{\varepsilon}}\Big[(1-\sigma(t))\|Du_{\varepsilon}(t)\|  + \sigma(t)\|D(u_{\varepsilon} + \varphi)(t)\|\Big]\dt \\
& \leq \int_0^Te^{-\frac{t}{\varepsilon}}\|Du_{\varepsilon}(t)\|\dt  + \int_0^T\|D(u_{\varepsilon} + \varphi)(t)\|\dt  < \infty. 
\end{align*}
It is then easy to see that $\partial_tv_{\varepsilon, \delta} \in L^2(\Omstar_T)$ and that the boundary and initial conditions are realized, since either $\zeta(0) = 0$ or $\varphi(0) = 0$ by assumption. Therefore, $v_{\varepsilon, \delta} \in K_{u_0}^{\varepsilon}$ and from the minimality of $u_{\varepsilon}$ we conclude that
\begin{align*}
\mathcal{F}_{\varepsilon}(u_{\varepsilon}) \leq \mathcal{F}_{\varepsilon}(v_{\varepsilon, \delta}) < \infty.
\end{align*}
Let us now rewrite the minimality condition with $v_{\varepsilon, \delta}$; applying further the convexity of the total variation, this becomes 
\begin{align*}
0 & \leq \int_0^T e^{-\frac{t}{\varepsilon}}\int_{\Omstar}\frac{1}{2}\left(|\partial_tu_{\varepsilon} + \delta\partial_t(e^{\frac{t}{\varepsilon}}\zeta\varphi)|^2 - |\partial_tu_{\varepsilon}|^2\right)\d\mu\dt \\
& \mskip+20mu + \int_0^Te^{-\frac{t}{\varepsilon}}\frac{1}{\varepsilon}\left(\|D(u_{\varepsilon}(t) + \delta e^{\frac{t}{\varepsilon}}\zeta(t)\varphi(t))\| - \|Du_{\varepsilon}(t)\|\right)\dt \\
& \leq \int_0^Te^{-\frac{t}{\varepsilon}}\int_{\Omstar}\left(\frac{1}{2}\delta^2|\partial_t(e^{\frac{t}{\varepsilon}}\zeta\varphi)|^2 + \delta\partial_tu_{\varepsilon}\partial_t(e^{\frac{t}{\varepsilon}}\zeta\varphi)\right)\d\mu \dt \\
& \mskip+20mu + \int_0^T\frac{\delta}{\varepsilon}\zeta(t)\Big(\|D(u_{\varepsilon} + \varphi)(t)\| - \|Du_{\varepsilon}(t)\|\Big)\dt.
\end{align*}
We multiply the preceding inequality by $\dfrac{\varepsilon}{\delta}$ and then let $\delta \searrow 0$. This yields 
\begin{align}\label{div-eps-delta-est}\begin{split}
0 & \leq \int_0^T e^{-\frac{t}{\varepsilon}}\int_{\Omstar}\varepsilon\partial_tu_{\varepsilon}\partial_t(e^{\frac{t}{\varepsilon}}\zeta\varphi)\d\mu\dt  + \int_0^T\zeta\Big(\|D(u_{\varepsilon} + \varphi)(t)\| - \|Du_{\varepsilon}(t)\|\Big)\dt \\
& = \int_0^T\zeta(t)\left[\int_{\Omstar}\partial_tu_{\varepsilon}\varphi\d\mu + \|D(u_{\varepsilon} + \varphi)(t)\| - \|Du_{\varepsilon}(t)\|\right]\dt \\
& \mskip+20mu + \varepsilon\int_0^T\int_{\Omstar}\Big[\zeta'\partial_tu_{\varepsilon}\varphi + \zeta\partial_tu_{\varepsilon}\partial_t\varphi\Big]\d\mu\dt.\end{split}
\end{align}
Notice that we can rewrite \eqref{div-eps-delta-est} as follows: 
\begin{align}\label{minrewritten}
\begin{aligned}
\int_0^T\zeta(t)\|Du_{\varepsilon}(t)\|\dt & \leq \int_0^T\zeta(t)\|D(u_{\varepsilon} + \varphi)(t)\|\dt \\
& \mskip+20mu + \int_0^T\int_{\Omstar}\zeta\partial_tu_{\varepsilon}\varphi\d\mu\dt \\
& \mskip+20mu + \varepsilon\int_0^T\int_{\Omstar}\Big[\zeta'\partial_tu_{\varepsilon}\varphi + \zeta\partial_tu_{\varepsilon}\partial_t\varphi\Big]\d\mu\dt,
\end{aligned}
\end{align}
which holds for any Lipschitz map $\zeta: (0, T) \to [0, 1]$ and any test function $\varphi \in L^1_w(0, T; \BV_0(\Om))$ with $\partial_t\varphi \in L^2(\Omstar_T)$, satisfying (\ref{phifinite}), and such that either $\zeta(0) = 0$ and $\varphi(0) \in L^2(\Omstar)$ or $\varphi(0) = 0$.

\subsection{Energy bounds}\label{subsec-E-bound}

This section will be devoted to the quest for uniform energy bounds to be satisfied by $\Feps$-minimizers $u_{\varepsilon} \in \mathcal{K}_{u_0}^{\varepsilon}$. Such bounds will then be the starting point to determine a converging subsequence as we pass to the limit as $\varepsilon \searrow 0$. To this aim, take a Lipschitz map $\zeta: (0, T) \to [0, 1]$ and define $[u_{\varepsilon}]_h^{u_0}$ according to (\ref{mollification}). First of all, Lemma \ref{mlemma2} implies $[u_{\varepsilon}]_h^{u_0} \in L^1_w(0, T; \BV(\Omstar))$; moreover, $[u_{\varepsilon}]_h^{u_0}(0) = u_0$ and therefore $\partial_t[u_{\varepsilon}]_h^{u_0} = \frac{1}{h}(u_0 - [u_{\varepsilon}]_h^{u_0}(0)) = 0$. Again by Lemma \ref{mlemma2} we infer
\begin{align*}
\int_0^T\zeta(t)\|D(u_{\varepsilon} - h\partial_t[u_{\varepsilon}]_h^{u_0})(t)\|\dt = \int_0^T\zeta(t)\|D[u_{\varepsilon}]_h^{u_0}(t)\|\dt
\end{align*}
for every $\zeta: (0, T) \to [0, 1]$ that is Lipschitz. This entitles us to set $\varphi \coloneqq -h\partial_t[u_{\varepsilon}]_h^{u_0}$ in (\ref{minrewritten}); observe that $\varphi$ satisfies all the requirements in (\ref{minrewritten}). Thus,
\begin{align}\label{5.8-mod}\begin{split}
 h\int_0^T\int_{\Omstar}\big[(\zeta + \varepsilon\zeta')\partial_tu_{\varepsilon}\partial_t[u_{\varepsilon}]_h^{u_0} &+ \varepsilon\,\zeta\,\partial_tu_{\varepsilon}\partial_{tt}[u_{\varepsilon}]_h^{u_0}\big]\d\mu\dt \\
&  \leq \int_0^T\zeta(t)\big[\|D[u_{\varepsilon}]_h^{u_0}(t)\| - \|Du_{\varepsilon}(t)\|\big]\dt \\
& \leq \int_0^T\zeta(t)\left[\left[\|Du_{\varepsilon}(t)\|\right]_h^{\|Du_0\|} - \|Du_{\varepsilon}(t)\|\right]\dt \\
&= -h\int_0^T\zeta(t)\partial_t\left[\|Du_{\varepsilon}(t)\|\right]_h^{\|Du_0\|}\dt.\end{split}
\end{align}
Above, we applied once more Lemma \ref{mlemma2}. Now, we divide both sides in \eqref{5.8-mod} by $h$ and manipulate the left-hand side to estimate the quantity $\partial_tu_{\varepsilon}\partial_{tt}[u_{\varepsilon}]_h^{u_0}$ as follows:
\begin{align*}
\partial_tu_{\varepsilon}\partial_{tt}[u_{\varepsilon}]_h^{u_0} &= \partial_t[u_{\varepsilon}]_h^{u_0}\partial_{\tt}[u_{\varepsilon}]_h^{u_0} + (\partial_tu_{\varepsilon} - \partial_t[u_{\varepsilon}]_h^{u_0})\partial_{tt}[u_{\varepsilon}]_h^{u_0} \\
& = \frac{1}{2}\partial_t|\partial_t[u_{\varepsilon}]_h^{u_0}|^2 + \frac{1}{h}|\partial_t[u_{\varepsilon}]_h^{u_0} - \partial_tu_{\varepsilon}|^2 \\
& \geq \frac{1}{2}\partial_t|\partial_t[u_{\varepsilon}]_h^{u_0}|^2.
\end{align*}
After substituting the above estimate in \eqref{5.8-mod}, we obtain
\begin{align}\label{twoways}
\begin{aligned}
\int_0^T\int_{\Omstar}\bigg[(\zeta + \varepsilon\zeta')\partial_tu_{\varepsilon}\partial_t[u_{\varepsilon}]_h^{u_0} & + \frac{\varepsilon}{2}\zeta\partial_t|\partial_t[u_{\varepsilon}]_h^{u_0}|^2\bigg]\d\mu\dt \\
&  \leq -\int_0^T\zeta(t)\partial_t[\|Du_{\varepsilon}(t)\|]_h^{\|Du_0\|}\dt.
\end{aligned}
\end{align}
We shall now employ \eqref{twoways}. First of all, let $\zeta \equiv 1$. Then 
\begin{align}\label{twoways-first}
\begin{aligned}
\int_0^T\int_{\Omstar}\partial_tu_{\varepsilon}\partial_t[u_{\varepsilon}]_h^{u_0}\dt & \leq -\int_0^T\partial_t[\|Du_{\varepsilon}(t)\|]_h^{\|Du_0\|}\dt \\
& \mskip+20mu - \frac{\varepsilon}{2}\int_0^T\int_{\Omstar}\partial_t|\partial_t[u_{\varepsilon}]_h^{u_0}|^2\d\mu\dt \\
& \leq \|Du_0\|.
\end{aligned}
\end{align}
In \eqref{twoways-first} we used the facts that
\begin{align*}
 \left[\|Du_{\varepsilon}(t)\|\right]_h^{\|Du_0\|}(0) & = \|Du_0\|, \\
 \partial_t[u_{\varepsilon}]_h^{u_0}(0) & = 0, \\
 \left[\|Du_{\varepsilon}(t)\|\right]_h^{\|Du_0\|}(T) & \geq 0,\\
 \partial_t[u_{\varepsilon}]_h^{u_0}|^2(T) & \geq 0.
\end{align*}

By passing to the limit as $h \searrow 0$ in \eqref{twoways-first} and taking into account that $\partial_t[u_{\varepsilon}]_h^{u_0} \to \partial_tu_{\varepsilon}$ in $L^2(\Omstar)$ as $h \searrow 0$ since $\partial_tu_{\varepsilon} \in L^2(\Omstar_T)$ by Lemma \ref{mlemma1}, we arrive at a uniform bound on the time derivative of $u_{\varepsilon}$:
\begin{align}\label{uniformtimeeps}
\int_0^T\int_{\Omstar}|\partial_tu_{\varepsilon}|^2\d\mu\dt \leq \|Du_0\|.
\end{align}
In a similar fashion to Section \ref{subsec:sequence}, this entails
\begin{align}\label{L2boundeps}
\begin{aligned}
\|u_{\varepsilon}\|_{L^2(\Omstar_T)}^2 & \leq T^2\|\partial_tu_{\varepsilon}\|_{L^2(\Omstar_T)}^2 + 2T\|u_0\|_{L^2(\Omstar)}^2 \\
& \leq T^2\|Du_0\| + 2T\|u_0\|_{L^2(\Omstar)}^2.
\end{aligned}
\end{align}
Eventually, \eqref{uniformtimeeps} yields, for any $0 \leq s < t \leq T$ 
\begin{align}\label{hoeldereps}
\begin{aligned}
\|u_{\varepsilon}(t) - u_{\varepsilon}(s)\|_{L^2(\Omstar)} & \leq \|\partial_tu_{\varepsilon}\|_{L^2(\Omstar_T)}\sqrt{|t-s|} \\
& \leq \sqrt{\|Du_0\|}\sqrt{|t - s|}.
\end{aligned}
\end{align}
Notice that (\ref{L2boundeps}) and (\ref{hoeldereps}) imply the uniform boundedness of the family of $\Feps$-minimizers $(u_{\varepsilon})_{\varepsilon > 0}$  both in $L^2(\Omstar_T)$ and $C^{0, \frac{1}{2}}([0, T]; L^2(\Omstar))$.

Let us now start again from \eqref{twoways}. For $0 \leq t_1 < t_2 \leq T$ we choose $\zeta = \zeta_{t_1, t_2}$, where
\begin{align*}
\zeta_{t_1, t_2} := \left\{\begin{array}{cl}
1, & \text{ if } t \in [0, t_1], \\ 
\dfrac{t_2 - t}{t_2 - t_1}, & \text{ if } t \in (t_1, t_2), \\ 
0, & \text{ it } t \in [t_2, T].
\end{array}\right. 
\end{align*}
With this choice of $\zeta$, we reformulate (\ref{twoways}) by integrating by parts at both sides. In doing so, we take into account that the left-hand side features no boundary terms, while the right-hand side will contain the quantity $\|Du_0\|$ as a boundary contribution. Thus, 
\begin{align*}
 \int_0^T\int_{\Omstar}\zeta_{t_1, t_2}(t)\partial_tu_{\varepsilon}\partial_t[u_{\varepsilon}]_h\d\mu\dt & \leq \|Du_0\| + \int_0^T\zeta_{t_1, t_2}'\bigg[[Du_{\varepsilon}(t)\|]_h^{\|Du_0\|} \\
 & \mskip+20mu + \int_{\Omstar}\left[\frac{\varepsilon}{2}|\partial_t[u_{\varepsilon}]_h^{u_0}|^2 - \varepsilon\partial_tu_{\varepsilon}\partial_t[u_{\varepsilon}]_h^{u_0}\right]\d\mu\bigg]\dt.
\end{align*}
We now let $h \searrow 0$ to find that 
\begin{align*}
& \int_0^T\int_{\Omstar}\zeta_{t_1, t_2}(t)|\partial_tu_{\varepsilon}|^2\d\mu\dt \leq \|Du_0\| + \int_0^T\zeta_{t_1, t_2}'\left[\|Du_{\varepsilon}(t)\| - \frac{\varepsilon}{2}\int_{\Omstar}|\partial_tu_{\varepsilon}|^2\d\mu\right]\dt.
\end{align*}
Since $|\partial_tu_{\varepsilon}|^2 \geq 0$, the last estimate brings us to
\begin{align}\label{L1BVboundwitheps}
\begin{aligned}
\int_{t_1}^{t_2}\|Du_{\varepsilon}(t)\|\dt & \leq (t_2 - t_1)\|Du_0\| + \frac{\varepsilon}{2}\int_{t_1}^{t_2}\int_{\Omstar}|\partial_tu_{\varepsilon}|^2\d\mu\dt \\
& \leq \left(t_2 - t_1 + \frac{\varepsilon}{2}\right)\|Du_0\|
\end{aligned}
\end{align}
for any $0 \leq t_1 < t_2 \leq T$. Here, we made use of \eqref{uniformtimeeps}.

\subsection{Passage to the limit}\label{subsec-limit}

Now that we have found suitable energy bounds from Section \ref{subsec-E-bound}, we shall pass to the limit as $\varepsilon \searrow 0$ in the sequence of $\mathcal{F}_{\varepsilon}$-minimizers $u_{\varepsilon}$ on $\Om_T$; this will result in the proof of Theorem \ref{maintheo1}. We notice that (\ref{uniformtimeeps}), (\ref{L2boundeps}) and (\ref{L1BVboundwitheps}) entail the boundedness of the family $(u_{\varepsilon})_{\varepsilon > 0}$ of $\mathcal{F}_{\varepsilon}$-minimizing functions in $L^2(\Omstar)$; correspondingly, the time derivatives $\partial_tu_{\varepsilon}$ are bounded in $L^2(\Omstar)$ and the total variations $t \mapsto \|Du_{\varepsilon}(t)\|$ are bounded in $L^1(0, T)$. We also observe that such behaviours are all uniform with respect to $\varepsilon \in (0, 1]$. Thus, we can apply  \cite[Theorem 1]{Simon} - see also Lemma \ref{compactness} - to infer the existence of a subsequence $\varepsilon_j \searrow 0$ - which we shall keep not relabelled as $\varepsilon$ - and also of a measurable function $u: \Omstar_T \to \R$ satisfying
\begin{align}\label{convergence4}
\begin{aligned}
\left\{\begin{array}{cl}
u_{\varepsilon} \longrightarrow u & \text{ strongly in } L^2(\Omstar_T), \\ 
u_{\varepsilon} \longrightarrow u & \text{ a.e. on } \Omstar_T, \\ 
u_{\varepsilon} \longweak u & \text{ weakly in } L^2(\Omstar_T), \\ 
\partial_tu_{\varepsilon} \longweak \partial_tu & \text{ weakly in } L^2(\Omstar_T).
\end{array} \right.
\end{aligned}
\end{align}
First of all, Lemma \ref{uinweak} ensures that $u \in L^1_w(0, T; \BV(\Omstar))$; taking into account the lower semicontinuity with respect to weak $L^2$-convergence and (\ref{uniformtimeeps}), there holds
\begin{align}\label{lowersemicont}
\int_0^{\infty}\int_{\Omstar}|\partial_tu|^2\d\mu\dt \leq \liminf_{\varepsilon \searrow 0}\int_0^{\infty}\int_{\Omstar}|\partial_tu_{\varepsilon}|^2\d\mu\dt \leq \|Du_0\|.
\end{align}
Also, the lower semicontinuity of the total variation on the time slices together with Fatou's lemma and \eqref{L1BVboundwitheps} allow us to estimate
\begin{align*}
\int_{t_1}^{t_2}\|Du(t)\|\dt & \leq \int_{t_1}^{t_2}\liminf_{\varepsilon \searrow 0}\|Du_{\varepsilon}(t)\|\dt \\
& \leq \liminf_{\varepsilon \searrow 0}\int_{t_1}^{t_2}\|Du_{\varepsilon}(t)\|\dt \\
& \leq (t_2 - t_1)\|Du_0\| < \infty.
\end{align*}
Let us rewrite this inequality with $t_1 = 0$ and $t_2 = T$. Then,
\begin{align}\label{finiteenergy}
0 \leq \int_0^T\|Du(t)\|\dt < \infty,
\end{align}
which proves  the finiteness of the left-hand side in \eqref{varineq}. Now combining (\ref{hoeldereps}) with $u_{\varepsilon}(0) = u_0$, we conclude - similarly to Lemma \ref{exlem} - that also $u(0) = u_0$ in the usual $L^2$-sense. Finally, the pointwise a.e. convergence from the second of \eqref{convergence4} yields
\begin{align*}
|u(x, t) - u_0(x)| \leq |u(x, t) - u_{\varepsilon}(x, t)| + |u_{\varepsilon}(x, t) - u_0(x)| \to 0 \text{ as } \varepsilon \searrow 0
\end{align*}
for $(\mu \otimes \L^1)$-a.e. $(x, t) \in (\Omstar \setminus \Om) \times (0, \infty)$. This means that $u(x, t) = u_0(x)$ $(\mu\otimes\L^1)$-a.e. on $(\Omstar\setminus\Omega)\times(0,T)$. To conclude our argument, we are left to show that the limit function $u$ accounts as a variational solution in the sense of Definition \ref{varsoldef}.

To see this, observe that by virtue of \eqref{finiteenergy} it is enough to take $v \in L^1_w(0, T; \BV_{u_0}(\Om))$ such that $\partial_tv \in L^2(\Omstar_T)$ and $v(0) \in L^2(\Omstar)$ with the extra requirement of finite $\BV$-energy, namely
\begin{align}\label{finiteenergy2}
\int_0^T\|Dv(t)\|\dt < \infty,
\end{align}
since otherwise \eqref{varineq}) would be trivial. Let $\vartheta \in (0, \frac{T}{2})$; we introduce a cutoff function $\zeta_\vartheta$ by setting
\begin{align*}
\zeta_{\vartheta}(t) \coloneqq\left\{\begin{array}{cl}
\dfrac{1}{\vartheta}t, & \text{ if } t \in [0, \vartheta), \\ 
1, & \text{ if } t \in [\vartheta, t - \vartheta],  \\ 
\dfrac{1}{\vartheta}(T - t), & \text{ if } t \in (T - \vartheta, T].
\end{array} \right.
\end{align*}
Now, fix $\varepsilon \in (0, 1]$ and set $\varphi \coloneqq v - u_{\varepsilon}$. Since $\varphi \in L^1_w(0, T; \BV_0(\Om))$ has time-derivative $\partial_t\varphi \in L^2(\Omstar_T)$ and $\varphi(0) \in L^2(\Omstar)$ satisfies \eqref{phifinite}, we are entitled to apply \eqref{minrewritten} with $\zeta \coloneqq \zeta_{\vartheta}$. This results in
\begin{align*}
 \int_0^T\|Du_{\varepsilon}(t)\|\dt &
\leq \int_0^T(1 - \zeta_{\vartheta}(t))\|Du_{\varepsilon}(t)\|\d\mu\dt \\ 
& \mskip+20mu + \int_0^T\int_{\Omstar}\zeta_{\vartheta}\partial_tu_{\varepsilon}(v - u_{\varepsilon})\d\mu\dt + \int_0^T\zeta_{\vartheta}(t)\|Dv(t)\|\dt \\
& \mskip+20mu + \varepsilon\int_0^T\int_{\Omstar}[\zeta_{\vartheta}'\partial_tu_{\varepsilon}(v - u_{\varepsilon}) + \zeta_{\vartheta}\partial_tu_{\varepsilon}\partial_t(v - u_{\varepsilon})]\d\mu\dt \\
& \coloneqq\I_{\varepsilon} + \II_{\varepsilon} + \III + \IV_{\varepsilon},
\end{align*}
with the respective interpretations of $\I_{\varepsilon}, \II_{\varepsilon}, \III, \IV_{\varepsilon}$ being obvious. If $\vartheta \geq \varepsilon$, then we can estimate $\I_{\varepsilon}$ by means of \eqref{L1BVboundwitheps} as follows: 
\begin{align*}
 \frac{1}{2\vartheta}\int_0^{\vartheta}\int_{\Omstar}|v - u_{\varepsilon}|^2\d\mu\dt & \leq \Bigg[\Bigg(\frac{1}{2\vartheta}\int_0^{\vartheta}\int_{\Omstar}|v - u_0|^2\d\mu\dt\Bigg)^{\frac{1}{2}} \\ 
& \mskip+100mu + \Bigg(\frac{1}{2\vartheta}\int_0^{\vartheta}\int_{\Omstar}|u_{\varepsilon} - u_0|^2\d\mu\dt\Bigg)^{\frac{1}{2}}\Bigg]^2 \\
& \leq \Bigg[\Bigg(\frac{1}{2\vartheta}\int_0^{\vartheta}\int_{\Omstar}|v - u_0|^2\d\mu\dt\Bigg)^{\frac{1}{2}} + \Bigg(\frac{1}{2\vartheta}\|Du_0\|\Bigg)^{\frac{1}{2}}\Bigg]^2.
\end{align*}
Above, we used \eqref{hoeldereps} with $s = 0$ to estimate the second term. Now, applying Fatou's Lemma together with the convergence $u_{\varepsilon} \to u$ $(\mu \otimes \L^1)$-a.e. on $\Omstar_T$, we get:
\begin{align*}
\liminf_{\varepsilon \searrow 0} -\frac{1}{2\vartheta}\int_{T - \vartheta}^T\int_{\Omstar}|v - u_{\varepsilon}|^2\d\mu\dt & \leq -\frac{1}{2\vartheta}\int_{T - \vartheta}^T\int_{\Omstar}\liminf_{\varepsilon \searrow 0}|v - u_{\varepsilon}|^2\d\mu\dt \\
& = -\frac{1}{2\vartheta}\int_{T - \vartheta}^T\int_{\Omstar}|v - u|^2\d\mu\dt.
\end{align*}
Let us then turn to $\II_\eps$. By the weak convergence $u_{\varepsilon} \weak u$ in $L^2(\Omstar_T)$, we can pass to the limit $\varepsilon \searrow 0$ to find 
\begin{align*}
\liminf_{\varepsilon \searrow 0}\II_{\varepsilon} & \leq \int_0^T\int_{\Omstar}\zeta_{\vartheta}\partial_tv(v- u)\d\mu\dt - \frac{1}{2\vartheta}\int_{T - \vartheta}^T\int_{\Omstar}|v - u|^2\d\mu\dt \\
& \mskip+20mu + \left[\left(\frac{1}{2\vartheta}\int_0^{\vartheta}\int_{\Omstar}|v - u_0|^2\d\mu\dt\right)^{\frac{1}{2}} + \left(\frac{1}{2\vartheta}\|Du_0\|\right)^{\frac{1}{2}}\right]^2.
\end{align*}
The treatment of $\IV_\eps$ is easy since $\partial_tu_{\varepsilon}$ and $u_{\varepsilon}$ are uniformly bounded in $L^2(\Omstar_T)$, so we immediately have $\IV_{\varepsilon} \to 0$ as $\varepsilon \searrow 0$. We now exploit the previous inequalities combined with the lower semicontinuity of the total variation, so that 
\begin{align}\label{prev+lsc}\begin{split}
 \int_0^T\|Du(t)\|\dt &  \leq \liminf_{\varepsilon \searrow 0}\int_0^T\|Du_{\varepsilon}(t)\|\dt \\
& \leq \int_0^T\zeta_{\vartheta}(t)\left[\int_{\Omstar}\partial_tv(v-u)\d\mu + \|Dv(t)\|\right]\dt + 3\vartheta\|Du_0\| \\
& \mskip+20mu + \Bigg[\left(\frac{1}{2\vartheta}\int_0^{\vartheta}\int_{\Omstar}|v - u_0|^2\d\mu\dt\right)^{\frac{1}{2}} + \left(\frac{1}{2\vartheta}\|Du_0\|\right)^{\frac{1}{2}}\Bigg]^2 \\
& \mskip+20mu - \frac{1}{2\vartheta}\int_{T - \vartheta}^T\int_{\Omstar}|v - u|^2\d\mu\dt.\end{split}
\end{align}
Notice that \eqref{prev+lsc} is realized whenever $\vartheta \in (0, \frac{T}{2})$, so we can pass to the limit as $\vartheta \searrow 0$ on the right-hand side. This yields 
\begin{align*}
\int_0^T\|Du(t)\|\dt & \leq \int_0^T\left[\int_{\Omstar}\partial_tv(v-u)\d\mu + \|Dv(t)\|\right]\dt \\
& \mskip+20mu + \frac{1}{2}\|v(0) - u_0\|_{L^2(\Omstar)}^2 - \frac{1}{2}\|(v-u)(T)\|_{L^2(\Omstar)}^2.
\end{align*}
In other words, $u$ is a variational solution on $\Omstar_T$ in the sense of Definition \ref{varsoldef}, and by virtue of  Lemma \ref{comparelem}, $u$ it is unique. As $T > 0$ was taken to be arbitrary, we can take $0 < T_1 < T_2 < \infty$ and denote by $u_1$ and $u_2$  respectively the unique variational solutions on $\Omstar_{T_1}$ and $\Omstar_{T_2}$. This, by repeating the localization procedure as in Section \ref{subsec:local}, brings us to the conclusion that $u_2$ is also a variational solution on the smaller cylinder $\Omstar_{T_1}$, which coincides with $u_1$ therein; namely, $u_1 = u_2$ on $\Omstar_{T_1}$. In doing so, we have just built a unique global variational solution. This concludes the proof of Theorem \ref{maintheo1}.

\section{Parabolic minimizers to the total variation flow}

\begin{definition}\label{def-par-min}
 We say that a measurable function $u:\Omega^*_{\infty}\to\R$ is a \textit{parabolic minimizer} to the total variation flow if and only if for any $T>0$ one has
 \begin{equation*}
  u\in L^1_w(0,T;BV_{u_0}(\Omega))\qquad\text{and}\qquad\int_0^T\|Du(t)\|(\Omstar)\d t<\infty
 \end{equation*}
together with the minimality condition
\begin{equation}\label{tvf-min}
 \int_0^T\left(\int_{\Omega^*}u\cdot\partial_t\varphi\d\mu+\|Du(t)\|(\Omstar)\right)\d t\le\int_0^T\|D(u+\varphi)(t)\|(\Omstar)\d t
\end{equation}
for all $\varphi\in\Lip(\Omstar_T)$ with $\supp \varphi \Subset \Om_T$.
\end{definition}

\begin{proposition}
 Let $u$ be a variational solution in the sense of Definition \ref{varsoldef}. Then, $u$ is a parabolic minimizer to the total variation flow.
\end{proposition}

\begin{proof}

 Let us fix $T>0$ and consider a test function $\varphi$ as in Definition \ref{def-par-min}. Observe that Theorem \ref{maintheo2} entails $\partial_t u\in L^2(\Omega^*_T)$.
 In order to prove \eqref{tvf-min}, we consider a map $v\in L^1_w(0,T;BV_{u_0}(\Omega))$ of the form $v=u+s\varphi$, where $s>0$ is a comparison function as in the variational inequality (\ref{varineq}). It is not restrictive to ask
 \begin{equation*}
  \int_0^T\|D(u+\varphi)(t)\|(\Omstar)\d t<\infty,
 \end{equation*}
because otherwise \eqref{tvf-min} would follow immediately.

Now, since $v(0)=u_0$, $\partial_t v\in L^2(\Omega^*_{\infty})$ and $u(T)=v(T)$, we first obtain the estimate 
\begin{align}\label{bound-var} 
\begin{aligned}
\int_0^T\|Du(t)\|(\Omstar)\d t \le\int_0^T\left[\int_{\Omega^*}\partial_t(u+s\varphi)s\varphi\d\mu+\|D(u+s\varphi)(t)\|(\Omstar)\right]\d t.
\end{aligned}
\end{align}
We integrate by parts in the first term at the right-hand side of \eqref{bound-var} and we exploit the fact that we do not get any boundary term, since $\varphi(0)=0=\varphi(T)$; in the second term, we use instead the convexity of the total variation. This yields
\begin{align*}
 \int_0^T\bigg[\int_{\Omega^*}s(u+s\varphi)\partial_t\varphi\d\mu & +\|Du(t)\|(\Omstar)\bigg]\d t \\
&  \le \int_0^T\Big[(1-s)\|Du(t)\|(\Omstar)+s\|D(u+\varphi)(t)\|(\Omstar)\Big]\d t.
\end{align*}
Then, we subtract the quantity
\begin{equation*}
 (1-s)\int_0^T\|Du(t)\|(\Omstar)\d t
\end{equation*}
from both sides in the previous inequality  and divide the resulting expression by $s>0$. Thus, there holds
\begin{equation}\label{sub-div-est}
\int_0^T\left[\int_{\Omega^*}(u+s\varphi)\partial_t\varphi\d\mu+\|Du(t)\|(\Omstar)\right]\d t\le\int_0^T\|D(u+\varphi)(t)\|(\Omstar)\d t.
\end{equation}
Passing to the limit as $s \searrow 0$ in \eqref{sub-div-est}, we eventually arrive at
\begin{align*}
\int_0^T\left[\int_{\Omega^*}u\partial_t\varphi\d\mu+\|Du(t)\|(\Omega)\right]\d t\le \int_0^T\|D(u+\varphi)(t)\|(\Omega)\d t.
\end{align*}
Now, since $\varphi\in\Lip_c(\Omega_T)$ was taken to be arbitrary, we have just proven that $u$ is a parabolic minimizer to the total variation flow, and the claim follows.

\end{proof}

\appendix
\section{Appendix: Compactness}\label{app-compactness}
\renewcommand{\thesection}{A}
\setcounter{theorem}{0}
\setcounter{equation}{0}


Here, we are going to provide a compactness result for sequences in $L^1_{w}(0,T;\BV(\Omega))$; this  has been used repeatedly in Sections \ref{subsec:sequence}--\ref{subsec-limit} to conclude with convergence of subsequences in $L^{1}$. Since with our choices for the definition of $L^1_w(0,T;\BV(\Omega))$ by means of Lipschitz derivations the readaptation of this Lemma to the metric setting turns out to be immediate and analogous to the Euclidean case, we are going to state the result with no proof, referring the reader to \cite[Section 8]{boegeleintv} and \cite[Theorem 1]{Simon} for the complete arguments.

\medskip

\begin{lemma}\label{compactness}
Let $\Omega\subset\X$ be an open set in a metric measure space $(\X,d,\mu)$ equipped with a doubling measure $\mu$ and satisfying a weak (1,1)--Poincar\'e inequality. Let $L^1_w(0,T;BV(\Omega))$ denote the weak $L^1$--$BV$ space as defined in Section \ref{sub-weak-par}.

If a sequence $(u_{j})_{j\in\mathbb{N}}\subset L_{w}^{1}(0,T;\BV(\Omega))$ satisfies
\begin{equation*}\label{compcond}
\sup_{j\in\mathbb{N}}\left[\Vert u_{j}\Vert_{L^{1}(\Omega_{T})}+\Vert \partial_{t}u_{j}\Vert_{L^{1}(\Omega_{T})}+\int_{0}^{T}\Vert Du_{j}(t)\Vert(\Omega)\d t\right]<\infty,
\end{equation*}
then $(u_{j})_{j\in\mathbb{N}}$ is relatively compact in $L^{1}(\Omega_{T})$.
\end{lemma}


\vspace{1cm}

{\bf Acknowledgements \& Funding}

This work was partially supported by the grant DFG-Project HA 7610/1-1 ``Existenz- und Regularit\"atsaussagen für parabolische Quasiminimierer auf metrischen Ma{\ss}r\"aumen'' and by the Academy of Finland. M. Collins expresses his gratitude to the Friedrich Naumann Foundation for Freedom that supported him through a postgraduate scholarship. He also thanks the Deparment of Mathematics and System Analysis at Aalto University for kindly receiving him in September 2019 and February 2020, while V. Buffa and C. Pacchiano Camacho thank the Department of Mathematics at FAU Erlangen for the kind hospitality during their visit in October 2019.

\bigskip







\phantomsection

\end{document}